\documentclass[journal, twocolumn]{IEEEtran}
\usepackage{graphicx}
\graphicspath{{images_note/}}
\usepackage{subcaption}
\usepackage[]{algorithm2e}
\usepackage{amsmath,amsthm}
\usepackage{verbatim}
\usepackage{amsfonts}
\usepackage{multirow}
\usepackage{changepage}
\usepackage{cite}
\usepackage{tabularx}
\usepackage{adjustbox}
\usepackage{mathtools}
\usepackage{amssymb}
\usepackage{soul}

\usepackage{color}

\DeclareMathOperator{\E}{\mathbb{E}}

\def\real{\mathbb{R}}

\def\natural{\mathbb{N}}
\DeclareMathOperator*{\argmax}{arg max}

\newcommand{\until}[1]{\{1,\dots, #1\}}

\newcommand{\map}[3]{#1: #2 \mapsto #3}

\newcommand\oprocendsymbol{\hbox{$\square$}}
\newcommand\oprocend{\relax\ifmmode\else\unskip\hfill\fi\oprocendsymbol}

\newcommand\bit[1]{\textit{\textbf{#1}}}

\def \mc {\mathcal}

\def \mr {\mathrm}

\newtheorem{theorem}{Theorem}
\newtheorem{proposition}{Proposition}
\newtheorem{lemma}{Lemma}
\newtheorem{corollary}{Corollary}

\newtheorem{remark}{Remark}

\newtheorem{definition}{Definition}

\makeatletter
\let\NAT@parse\undefined
\makeatother
\usepackage[square, sort&compress, numbers]{natbib}

\begin{document}

\title{Incentivizing Collaboration in Heterogeneous Teams via Common-Pool Resource Games}
\author{Piyush Gupta \hspace{0.5in} Shaunak D. Bopardikar  \hspace{0.5in} Vaibhav Srivastava
\thanks{This work has been supported by NSF Award IIS-1734272. A preliminary version of this work~\cite{gupta2019achieving} was presented at the 58th Conference on Decision and Control. We expand on our work in~\cite{gupta2019achieving} by providing detailed proofs and analytic upper bound on the measures of inefficiency for the unique PNE.}
\thanks{Piyush Gupta (guptapi1@msu.edu), Shaunak D.~Bopardikar (shaunak@egr.msu.edu) and Vaibhav Srivastava (vaibhav@egr.msu.edu) are with Department of Electrical and Computer Engineering, Michigan State University, East Lansing, Michigan, 48824, USA.}}

\maketitle

\begin{abstract}
We consider a team of heterogeneous agents that is collectively responsible for servicing, and subsequently reviewing,  a stream of homogeneous tasks. Each agent has an associated mean service time and a mean review time for servicing and reviewing the tasks, respectively.
Agents receive a reward based on their service and review admission rates. The team objective is to collaboratively maximize the number of ``serviced and reviewed" tasks. We formulate a Common-Pool  Resource  (CPR)  game and design utility functions to incentivize collaboration among heterogeneous agents in a decentralized manner. We show the existence of a unique Pure Nash Equilibrium (PNE), and establish convergence of best response dynamics to this unique PNE. Finally, we establish an analytic upper bound on three measures of inefficiency of the PNE, namely the price of anarchy, the ratio of the total review admission rate, and the ratio of latency, along with an empirical study.


\end{abstract}

\begin{IEEEkeywords}
Best response potential, CPR game, Price of anarchy, PNE, Team collaboration, Utility design.  
\end{IEEEkeywords}

\IEEEpeerreviewmaketitle

\section{Introduction}
As we become more connected around the globe, team collaboration becomes a necessity to produce results. Although modern workplaces endeavor to be social and collaborative affairs, most workplaces fail to solve the conundrum of achieving efficient collaboration among diverse, dynamic, and dispersed team-members~\cite{haas2016secrets}. An effective collaboration requires each team-member to efficiently work on their tasks while backing up other team-members by  monitoring and providing feedback. Such team backup behavior improves team performance by mitigating the lack of certain skills in some team-members. 
Often times, lack of incentives to backup other members results in team-members operating individually and in a poor team performance. Therefore, for effective team performance, it is imperative to design appropriate incentives that facilitate collaboration among the agents without affecting their individual performance. 

Depending on the hierarchical structure, organizations are often distinguished as either mechanistic (bureaucratic), or  organic (professional)~\cite{burns2005mechanistic}. 
While mechanistic  organizations  are  characterized by a rigid hierarchy, high levels of formalization, and centralized decision making, organic organizations are flexible with weak  or  multiple  hierarchies, and have low  levels of formalization~\cite{lunenburg2012mechanistic}. The flexibility in  organic organizations leads to decentralized decision-making, and therefore, enables quick and easy reaction to changes in the environment. Hence, organic organizations cope best with the unpredictable and unstable environments that surrounds them, as compared to mechanistic organizations which are appropriate in stable environments and for routine tasks.
For organic organizations, which lack centralized decision-making authority, it is essential to incentivize collaboration among heterogeneous team-members to achieve efficient team performance. 

CPR games~\cite{keser1999strategic,hota2016fragility} is a class of resource sharing games in which players jointly manage a common pool of resource and make strategic decisions to maximize their utilities. In this paper, we design incentives for the heterogeneous agents to facilitate aforementioned team backup behavior. In particular, we connect the class of problems involving human-team-supervised autonomy~\cite{goodrich2007using} with the CPR games, and design utilities that yield the desired behavior. Within the queueing theory paradigm that has been used to study these problems~\cite{JP-VS-etal:12t,  VS-RC-CL-FB:11zc,PG-VS:18d}, we show that CPR games provide a formal framework to analyze and design organic teams. Specifically,
utilizing the CPR framework allows us to incentivize team collaboration among heterogeneous agents in a decentralized manner, i.e., efficient social utility is achieved despite self-interested actions of the individuals. 

Game-theoretic approaches have been utilized for problems in distributed control~\cite{marden2018game}, wherein the overall system is driven to an efficient equilibrium strategy that is close to the social optimum
through an appropriate design of utility functions~\cite{arslan2007autonomous}.
Price of Anarchy (PoA)~\cite{basar1999dynamic} is often used to characterize efficiency of the equilibrium strategies in a game. Associated analysis techniques
utilize smoothness property of the utility functions~\cite{roughgarden2009intrinsic}, leverage submodularity of the welfare function~\cite{marden2014generalized}, or solve an auxiliary optimization problem~\cite{deori2018price,paccagnan2019utility}. These approaches do not immediately apply to our setup.
Instead, we follow a new line of analysis to obtain bounds on PoA by constructing a homogeneous CPR game, {for which we show that the equilibrium strategy is also the social optimum (PoA=1),} and relating its utility to the original game.

Human-team-supervised autonomy is a class of motivating problems for our setup. Queueing theory has emerged as a popular paradigm to study these problems~\cite{JP-VS-etal:12t, VS-RC-CL-FB:11zc,PG-VS:18d}. However, these works predominantly consider a single human operator. There have been limited studies on human-team-supervised autonomy. These  include simulation based studies~\cite{gao2014modeling}, ad hoc design~\cite{hong2016human}, or non-interacting operators~\cite{srivastava2014knapsack}. Here, we focus on a game-theoretic approach to study one of the key features of the human-team-supervised autonomy: the team backup behavior, which  refers to the extent to which team-members help each other perform their roles~\cite{mekdeci2009modeling}. 



We model team backup behavior in the following way. We consider an unlimited supply of tasks from which each team-member may admit tasks for servicing at a constant rate.  We assume that each serviced task is stored in a common review pool for a second review. Each team-member can choose to spend a fraction of their time to review tasks from the common review pool and thereby provide a backup. {In our setup, any agent can review tasks from the common review pool, independent of who serviced the task. Therefore, any agent that services tasks can also participate in the review process without impacting the quality of review process. This is sensible in scenarios in which, for example, the review process involves performing a quality check using machines or verification through software.} Without any incentives, members may not choose to review the tasks as it may affect their individual performance. We focus on design of incentives, within the CPR game formalism, to facilitate team backup behavior. 


While we use human-team-supervised autonomy as a motivating example, our problem formulation can be applied to broad range of problems involving tandem queues~\cite{le1997theory}, where servicing and reviewing of tasks can be considered as the subsequent stages of the queueing-network. 
 Tandem queues are utilized to design efficient systems to study problems such as resource allocation, inventory management, process optimization, and quality control~\cite{thomopoulos2012fundamentals}. Existing game theoretic approaches~\cite{altman2005applications,xia2014service} to service rate control in tandem queues assume that a single server is present at each stage of the tandem queue and each server has its 
independent resources. In contrast, in our setup, multiple heterogeneous agents allocate their time at different stations based on their skill-sets and maximize the system throughput. Additionally, our mathematical techniques are applicable to many problems involving dual-screening process.  For example, in human-in-the-loop systems which are pervasive in areas such as search-and-rescue, semi-autonomous vehicle systems, surveillance, etc., humans often supervise (review) the actions (service) performed by the autonomous agents. In such settings, our framework incentivizes collaboration among heterogeneous agents.

 Our CPR formulation has features similar to the CPR game studied in~\cite{hota2016fragility,hota2018controlling}. In these works, authors utilize prospect theory to capture the risk aversion behavior of the players investing into a fragile CPR~\cite{ostrom1994rules} that fails if there is excessive investment in the CPR. In the case of CPR failure, no player receives any return from the CPR. While our  design of the common review pool is similar to the fragile CPR, our failure model incorporates the constraint that only serviced tasks can be reviewed. In contrast to the agent heterogeneity due to prospect-theoretic risk preferences in~\cite{hota2016fragility}, heterogeneity in our model arises due to differences in the agents' mean service and review times.

 The major contributions of this work are fivefold. First, we present a novel formulation of team backup behavior and design incentives, within the CPR game formalism, to facilitate such behavior ({Section~\ref{Problem Formulation}}). 
 Second, we show that there exists a unique PNE for the proposed game ({Section~\ref{Existence of PNE}}). Third, we show that the proposed game is a best response potential game as defined in~\cite{voorneveld2000best}, for which both sequential best response dynamics~\cite{dubey2006strategic} and simultaneous best reply dynamics~\cite{jensen2009stability} converge to the PNE ({Section~\ref{Convergence}}). Thus, the best response of self-interested agents in a decentralized team converge to the PNE. Fourth, we provide the structure of the social welfare solution ({Section~\ref{Social}}) and numerically quantify ({Section~\ref{Numerical Illustrations}}) different measures of the inefficiency for the PNE, namely the PoA, the ratio of the total review admission rate (TRI), and the ratio of latency (LI), as a function of a measure of heterogeneity.  While PoA is a widely used inefficiency metric, we define TRI and LI as other relevant measures for our setup based on the total review admission rate and latency (inverse of throughput), respectively.
 Finally, we provide an analytic upper bound for all three measures of the inefficiency (Section~\ref{Social}).

\section{Background and Problem Formulation} \label{Problem Formulation}
In this section, we describe the problem setup and formulate the problem using a game-theoretic framework. We also present some definitions that will be used in the paper. 
\subsection{Problem Description}{\label{Problem Setup}}
\begin{figure}
	\centering
	\includegraphics[width=0.75\linewidth, height=0.75\linewidth, keepaspectratio]{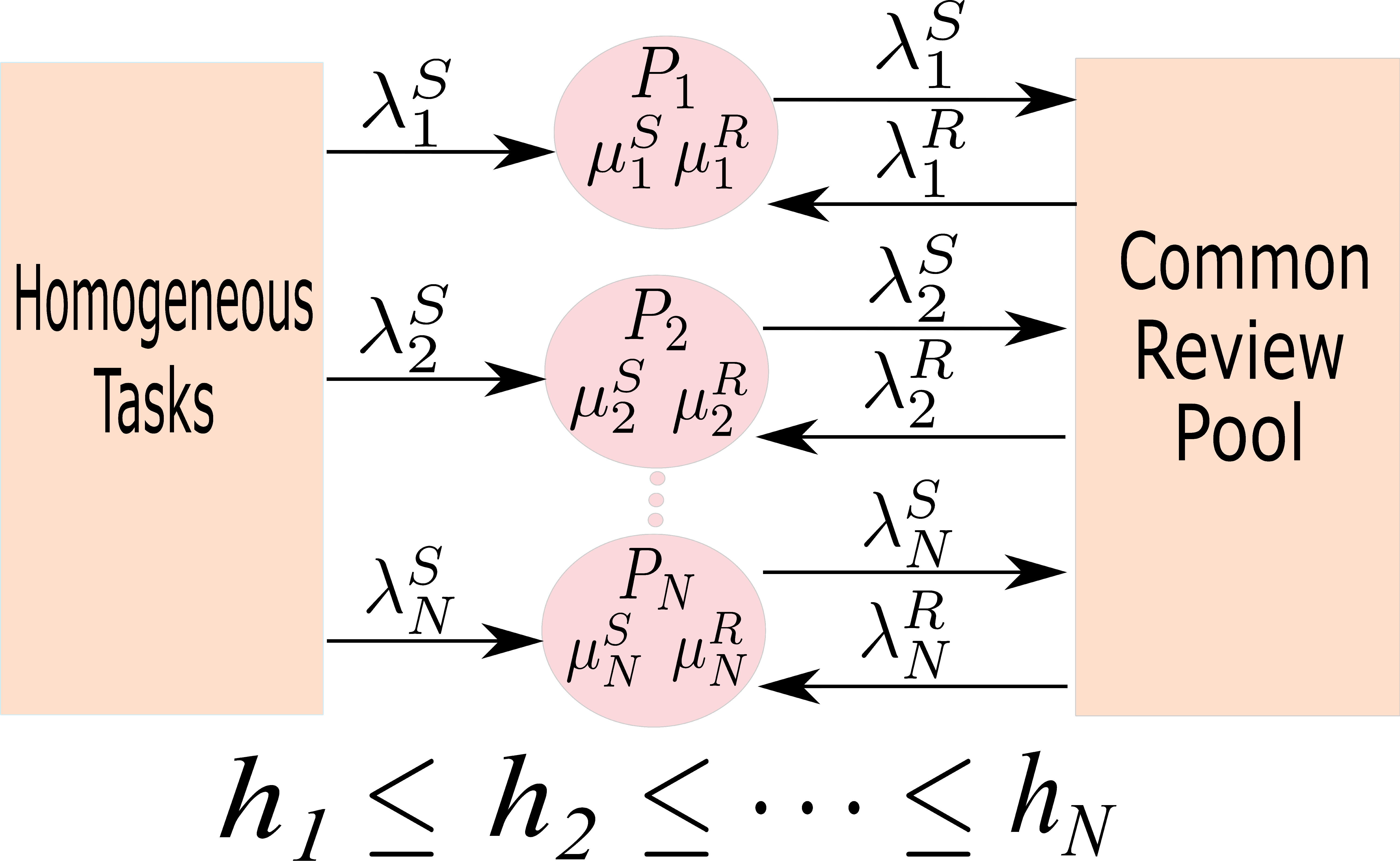}
    \caption{\footnotesize Player $i$ devotes her time to service homogeneous tasks (at a constant service admission rate $\lambda_i^S$) while reviewing serviced tasks from the common review pool (at a constant review admission rate $\lambda_i^R$). The maximum admission rate for player $i$ for servicing and reviewing the tasks is given by $\mu_i^S$ and $\mu_i^R$, respectively.}
    \label{fig:Model}
\end{figure}

We consider a heterogeneous team of $N\in \natural$ agents  tasked with servicing a stream of homogeneous tasks. These agents could be autonomous systems or human operators. Each task, after getting serviced by a team-member, gets stored in a common review pool for a second review. This second review is a feedback process in which any team-member can re-examine the serviced task from the common review pool for performance monitoring and quality assurance purposes.  Each agent $i \in \mc{N}= \until{N}$ may choose to spend a portion of her time to review the tasks from the common pool while spending her remaining time to service the incoming tasks. We consider heterogeneity among the operators due to the difference in their level of expertise and skill-sets in servicing and reviewing the tasks. This heterogeneity is captured by the average service time $(\mu_i^S)^{-1} \in \mathbb{R}_{>0}$ and average review time $(\mu_i^R)^{-1} \in \mathbb{R}_{>0}$ spent by operator $i \in \mc{N}$ on servicing and reviewing a task, respectively.


Let $\lambda_i^S \in [0, \ \mu_i^S]$ and $\lambda_i^R \in [0, \ \mu_i^R]$ be the deterministic service and review admission rates, i.e., the rates at which agent $i$ chooses to admit tasks for servicing and reviewing, respectively. 
Each agent $i$ can choose their service and review admission rate independent of other agents. The range of $\lambda_i^S$ and $\lambda_i^R$ have been chosen to satisfy the stability conditions (including marginal stability) for the service and review queues for operator $i \in \mc{N}$~\cite[Chapter 8]{cassandras2009introduction}. 



Suppose agent $i$ selects $\lambda_i^S$ and $\lambda_i^R$ as their service and review admission rates, then 
\[
\frac{\lambda_i^S}{\mu_i^S} + \frac{\lambda_i^R}{\mu_i^R} \le 1,
\]
where $\frac{\lambda_i^S}{\mu_i^S}$ (respectively, $\frac{\lambda_i^R}{\mu_i^R}$) is the average time the agent spends on servicing (respectively, reviewing) the tasks within a unit time. Thus, if the agent has selected a review admission rate $\lambda_i^R$, then the service admission rate satisfies
\begin{equation}{\label{eq:1}}
 \lambda_i^S \le  {\mu_i^S}-h_i\lambda_i^R,
\end{equation}
where $h_i:=\frac{\mu_i^S}{\mu_i^R}$ is the heterogeneity measure for the player $i$.






We consider self-interested agents  that receive  a  utility  based on their service and review admission rates. Hence, we will assume that agents operate at their maximum capacity, and equality holds in~\eqref{eq:1}. Fig.~\ref{fig:Model} shows the schematic of our problem setup. Note that only serviced tasks are available for review, and therefore,
\begin{equation}{\label{eq:2}}
   \sum_{i=1}^{N}\lambda_i^R \leq \sum_{i=1}^{N}\lambda_i^S.
\end{equation}
By substituting \eqref{eq:1} in \eqref{eq:2}, we obtain,
\begin{align}{\label{eq:3}}
    \sum_{i=1}^{N}a_i\lambda_i^R \leq \sum_{i=1}^{N}\mu_i^S,
\end{align}
where $ a_i:=(1+h_i)$. Eq.~\eqref{eq:3} represents the system constraint on the review admission rates chosen by agents. 

We are interested in incentivizing collaboration among the agents for the better team performance. Towards this end, we propose a game-theoretic setup defined below. 


\subsection{A Common-Pool Resource Game Formulation}
We now formulate our problem as a Common-Pool Resource (CPR) game. Henceforth, we would refer to each agent as a player. 
{A maximum service admission rate $\mu_i^S$ and a maximum review admission rate $\mu_i^R$ are associated with each player $i$,} based on her skill-set and level of expertise. Without loss of generality, let the players be labeled in increasing order of their heterogeneity measures, i.e., $h_1 \le \cdots \le h_N$. 

Let $S_i:=[0, \ \mu_i^R]$ be the strategy set for each player $i$, from which the player chooses her review admission rate for reviewing the tasks from the common review pool. {Since we have assumed~\eqref{eq:1} holds with equality, once} player $i$ decides her review admission rate $\lambda_i^R \in S_i$, her service admission rate for servicing the tasks $\lambda_i^S$ is given by the right hand side of~\eqref{eq:1}. Let $S=\prod_{i \in \mc{N}}S_i$ be the joint strategy space of all the players, where $\prod$ denotes the Cartesian product. Furthermore, we define $S_{-i}=\prod_{j \in \mc{N},j \neq i}S_j$ as the joint strategy space of all the players except player $i$. 


 For brevity of notation, we denote the total service admission rate and the total review admission rate by $\lambda_T^S= \sum_{i=1}^{N}\lambda_i^S$ and $\lambda_T^R= \sum_{i=1}^{N}\lambda_i^R$, respectively. Similarly, $\mu_T^S= \sum_{i=1}^{N}\mu_i^S$ and $\mu_T^R= \sum_{i=1}^{N}\mu_i^R$ denote the aggregated sum of the maximum service admission rates and maximum review admission rates of all the players, respectively. 
 
 Let $x \in \real$, defined by 
  \begin{equation}{\label{eq:x}}
     x = \lambda_T^S-\lambda_T^R=\mu_T^S-\sum_{i=1}^{N}{a_i}\lambda_i^R,
 \end{equation}
 be the slackness parameter for system constraint~\eqref{eq:3}. The constraint~\eqref{eq:3} is violated for negative values of $x$, i.e., when total review rate exceeds the total service rate. In
such an event, some players commit to review more tasks than that are available in the common review pool. The slackness parameter characterizes the gap between the total service admission rate and the total review admission rate for all the players.  In order to maximize high quality team throughput, i.e., the number of tasks that are both serviced and reviewed, we seek to incentivize the team  to operate close to $x=0$.

Each player $i$ receives a constant reward $r^S \in \real_{>0}$ for servicing each task. Hence, the service utility $u_i^S : S_i \mapsto \real_{>0}$ for player $i$ servicing the tasks at the service admission rate $\lambda_i^S$ is given by:
\begin{equation}{\label{eq:4}}
   u_i^S= \lambda_i^Sr^S.
\end{equation}
To incentivize collaboration among the agents, we design the review utility $u_i^R: S \mapsto \real_{>0}$ received for reviewing the tasks from the common review pool using two functions: a rate of return, $r^R : S \mapsto \real_{>0}$ for each reviewed task and a constraint probability $p: S \mapsto [0,1]$ of the common review pool. The constraint probability $p$ is a soft penalty on the violation of system constraint~\eqref{eq:3}.

We model the rate of return $r^R$ and the constraint probability $p$ in terms of the strategy of all the players through slackness parameter $x$.
 Furthermore, we assume that  $r^R$ is strictly decreasing in $x$. Therefore, {for each $x \in [0,\ \mu_T^S]$}, system constraint~\eqref{eq:3} is satisfied, and the rate of return is maximized at $x=0$. {The rate of return can be interpreted as the perks that the employer provides to all the employees for high quality service.} For example, an employer generates higher revenue based on the high quality throughput of her company, {i.e., based on the number of ``serviced and reviewed" tasks,
 which she redistributes among her employees as  perks  as per  their  contribution  to  the review process. Highest quality throughput is achieved by the company when the team efficiently reviews all the serviced tasks, i.e., when $x=0$.}

We introduce the constraint probability $p$ as a soft penalty of the violation of system constraint~\eqref{eq:3}, and therefore, we let $p=1$ if the constraint gets violated, i.e. when $x<0$. We assume that the constraint probability $p$ is non-increasing in $x$, 
and approaches 1 as $x$ approaches $0$. The class of sharply decreasing exponential functions, $p(\lambda_i^R, \lambda_{-i}^R)=\exp(-Ax),$ where $x \in [0, \ \mu_T^S]$ and $A \in \mathbb{R}_{>0}$, can be a good choice to effectively model the system constraint. If the constraint is violated with probability $p$, then $u_i^R=0$ for each player $i$.
Therefore, we define the utility $u_i^R$ by
\begin{equation}{\label{eq:5}}
    u_i^R(\lambda_i^R, \lambda_{-i}^R) = \begin{cases}
    0, & \!\!\!\!\!\!\!\!\!\!\!\!\!\!\!\!\!\  \text{with probability } p(\lambda_i^R,\lambda_{-i}^R),\\
    \lambda_i^Rr^R(\lambda_i^R,\lambda_{-i}^R), & \text{otherwise}.\\
    \end{cases}
\end{equation}

Let $u_i(\lambda_i^R, \lambda_{-i}^R)= u_i^S+ u_i^R$ be the total utility of player $i \in \mc{N}$. Each player $i$ tries to maximize her expected utility $\tilde u_i : S \mapsto \mathbb{R}$ defined by
\begin{align}{\label{eq:6}}
        \tilde u_i(\lambda_i^R, \lambda_{-i}^R) &= \E[u_i^S (\lambda_i^R, \lambda_{-i}^R) + u_i^R (\lambda_i^R, \lambda_{-i}^R)]\nonumber \\ 
        &= \lambda_i^Sr^S  +\lambda_i^Rr^R(\lambda_i^R,\lambda_{-i}^R)(1-p(\lambda_i^R,\lambda_{-i}^R)),
\end{align}
where the expectation is computed over the constraint probability $p$. Since $r^R$ and $p$ depend on the review admission rates of all the players only through the slackness parameter $x$, with a slight abuse of notation, we express $r^R(\lambda_i^R,\lambda_{-i}^R)$ and $p(\lambda_i^R,\lambda_{-i}^R)$ by $r^R(x)$ and $p(x)$, respectively. Substituting~\eqref{eq:1} in~\eqref{eq:6}, we get:
\begin{align}{\label{eq:7}}
   \tilde u_i&= \mu_i^Sr^S +\lambda_i^R\left[r^R(x)(1-p(x)) -h_ir^S\right] \nonumber \\
   &=: \mu_i^Sr^S +\lambda_i^Rf_i(x),
\end{align}
where $f_i: S \mapsto \mathbb{R}$ is defined by
\begin{equation}\label{eq:def-fi}
    f_i(\lambda_i^R,\lambda_{-i}^R)=f_i(x)=r^R(x)(1-p(x)) -h_ir^S.
\end{equation}

The function $f_i$ is the incentive for player $i$ to review the tasks. Note that player $i$  will choose a non-zero $\lambda_i^R$ if and only if she has a positive incentive to review the tasks,  i.e., $f_i(x) >0$. Otherwise, player $i$ drops out without reviewing any task ($\lambda_i^R=0$) and focuses solely on servicing of tasks ($\lambda_i^S=\mu_i^S$), thereby maximizing her expected utility given by $\tilde u_i=\mu_i^Sr^S$.

In the following, we will refer to the above CPR game by $\Gamma=(\mc{N},\{S_i\}_{i \in \mc{N}},\{\tilde u_i\}_{i \in \mc{N}})$. 
In this paper, we are interested in equilibrium strategies for the players that constitute a PNE defined below. 
\medskip
\begin{definition}[\bit{Pure Nash Equilibrium}] \label{definition1}
A PNE is a strategy profile ${\lambda^R}^*=\{{\lambda_i^R}^*\}_{i \in \mc{N}} \in S$, such that for each player $i \in \mc{N}$, $\tilde u_i({\lambda_i^R}^*,{\lambda_{-i}^R}^*) \geq \tilde u_i(\lambda_i^R,{\lambda_{-i}^R}^*)$, for any $\lambda_i^R \in S_i$.
\end{definition}
Let $\map{b_i}{S_{-i}}{S_i}$ defined by 
\[
b_i(\lambda_{-i}^R) \in \argmax_{\lambda_i^R \in S_i} \; \tilde u_i (\lambda_i^R, \ \lambda_{-i}^R), 
\]
be a \emph{best response} of player $i$ to the review admission rates of other players $\lambda_{-i}^R$. A PNE exists if and only if there exists an invariant strategy profile, ${\lambda^R}^*=\{{\lambda_i^R}^*\}_{i \in \mc{N}} \in S$, such that ${\lambda_i^R}^*=b_i({\lambda_{-i}^R}^*)$,  for each $i \in \mc{N}$.


\section{Existence and Uniqueness of PNE}\label{Existence of PNE}

In this section, we study the existence and uniqueness of the PNE for the CPR game $\Gamma$
under the system constraint~\eqref{eq:3}. Each player $i \in \mc{N}$ chooses a review admission rate from her strategy set $S_i=[0, \ \mu_i^R]$ and receives an expected utility $\tilde u_i$ given by~\eqref{eq:7}. For any given $\lambda_{-i}^R \in S_{-i}$, we obtain an upper bound $\overline{\lambda}_i^R : S_{-i} \mapsto S_i$ on $\lambda_i^R$ defined by 
\begin{equation*}{\label{eq:9}}
            \!\!\!\!\!\!\ \overline{\lambda}_i^R = \begin{cases}
            0, & \text{if } \Lambda_i < 0,\\
            \Lambda_i , & \!\!\  \text{if } 0\leq \Lambda_i \leq \mu_i^R,\\
            \mu_i^R, & \text{if } \Lambda_i > \mu_i^R,
            \end{cases}
        \end{equation*}
where $\Lambda_i:= \frac{\mu_T^S - \sum_{j \in \mc{N}, j \neq i}a_j\lambda_j^R}{a_i}$, such that for $\lambda_i^R \in [0, \  \overline{\lambda}_i^R) \subset S_i$, constraint~\eqref{eq:3} is automatically satisfied, and for $\lambda_i^R \in (\overline{\lambda}_i^R, \ \mu_i^R] \subset S_i$,  constraint~\eqref{eq:3} is violated. For $\lambda_i^R=\overline{\lambda}_i^R$, constraint~\eqref{eq:3} is satisfied if $\overline{\lambda}_i^R \in (0, \ \mu_i^R]$, and is violated if $\overline{\lambda}_i^R=0$. 

\begin{figure}
\centering
\begin{subfigure}[b]{0.32\linewidth}
	    \centering
        \includegraphics[width=1.1\linewidth, height=1.1\linewidth, keepaspectratio]{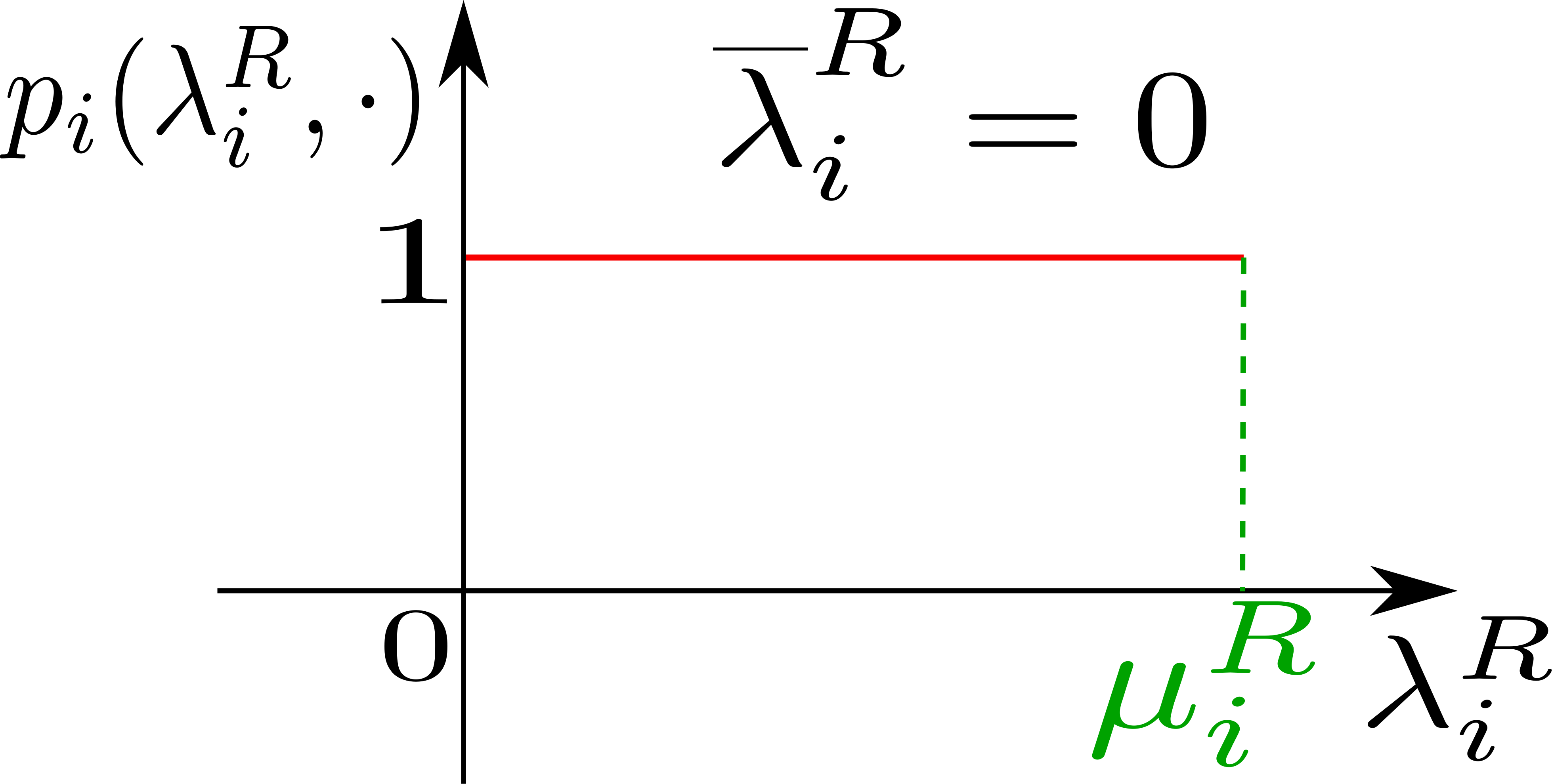}
        \caption{}
        \label{a}
    \end{subfigure}
\begin{subfigure}[b]{0.32\linewidth}
	    \centering
        \includegraphics[width=1.1\linewidth, height=1.1\linewidth, keepaspectratio]{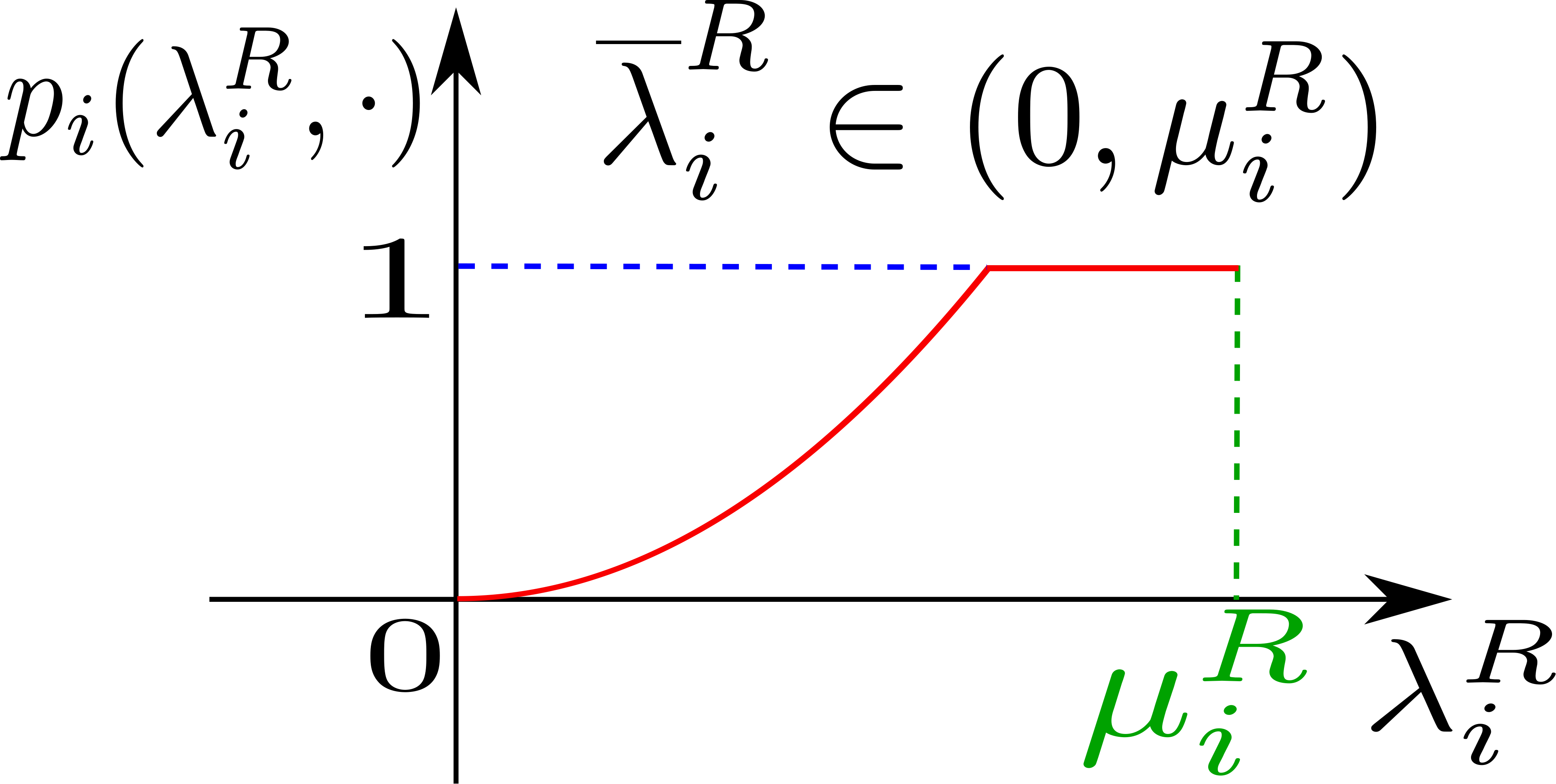}
        \caption{}
        \label{b}
\end{subfigure}
\begin{subfigure}[b]{0.32\linewidth}
	    \centering
        \includegraphics[width=1.1\linewidth, height=1.1\linewidth, keepaspectratio]{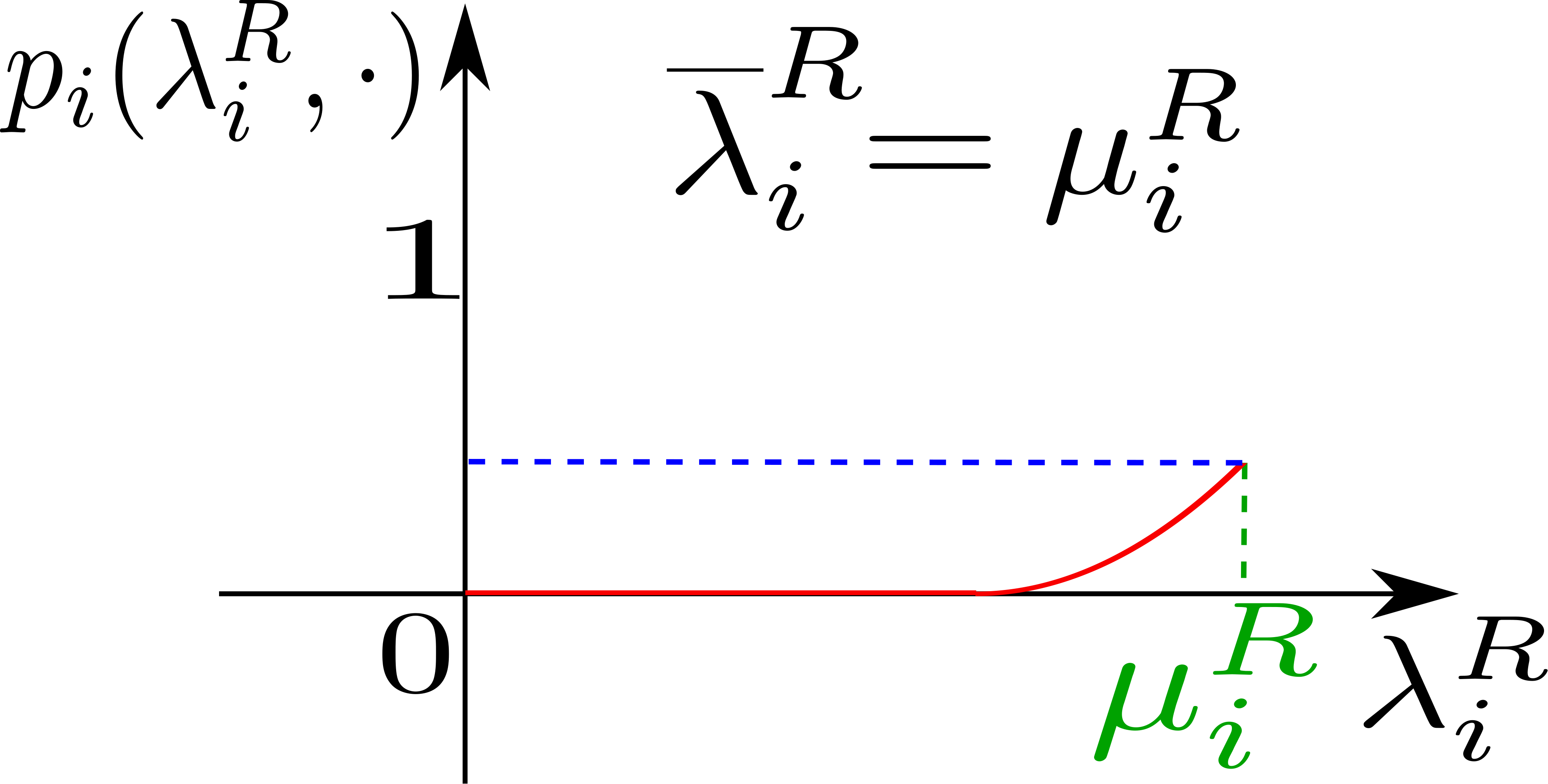}
        \caption{}
        \label{c}
\end{subfigure}
    \caption{\footnotesize Constraint probability of player $i$ as $\overline{\lambda}_i^R$ varies from $0$ to $\mu_i^R$. a) For $\overline{\lambda}_i^R=0$, $p_i(\lambda_i^R, \cdot)=1, \ \forall \lambda_i^R \in S_i$, 
    b) for $\overline{\lambda}_i^R \in (0, \ \mu_i^R)$, \ $ p_i(\lambda_i^R, \cdot)$ is convex for $\lambda_i^R \in [0, \ \overline{\lambda}_i^R)$, with $ p_i(\lambda_i^R, \cdot) \mapsto 1$ as $\lambda_i^R \mapsto \overline{\lambda}_i^R$, and $p_i(\lambda_i^R, \cdot)=1, \ \forall \lambda_i^R \in [\overline{\lambda}_i^R,\ \mu_i^R]$, and c) for $\overline{\lambda}_i^R=\mu_i^R$, $p_i(\lambda_i^R, \cdot)$ is convex in $\lambda_i^R$ and $ p_i(\lambda_i^R, \cdot) < 1, \ \forall \lambda_i^R \in S_i$.}
    \label{fig:Picture2}
\end{figure}

We study the properties of game $\Gamma$ under following assumptions. Recall that $x = \lambda_T^S-\lambda_T^R=\mu_T^S-\sum_{i=1}^{N}{a_i}\lambda_i^R$.

\begin{itemize}
    \item [(A1)] For a given $\lambda^R_{-i} \in S_{-i},\ i \in \mc N$, we assume that the rate of return $r^R(\lambda_i^R, \cdot)$ for reviewing the tasks is continuously differentiable, strictly increasing and strictly concave for $\lambda_i^R \in S_i$, with $r^R(0,0)=0$. Equivalently, $x \mapsto r^R(x)$ is continuously differentiable, strictly decreasing and strictly concave for $x \in [0, \ \mu_T^S] $, with $r^R(\mu_T^S)=0$.
    \item [(A2)] For a given $\lambda^R_{-i} \in S_{-i},\ i \in \mc N$, we assume that the constraint probability $p(\lambda^R_i, \cdot)$ is (i) continuous on $S_i$; (ii) is 
    continuously differentiable,  non-decreasing and convex  for $ \lambda^R_i \in (0, \ \overline{\lambda}_i^R) \subset S_i$; and (iii)
     is equal to $1$, for $\lambda^R_i \in (\overline{\lambda}_i^R, \ \mu_i^R]$. See Fig.~\ref{fig:Picture2} for an illustration. Equivalently, $p(x)$ is continuously differentiable, non-increasing and convex for $x \in (0, \ \mu_T^S]$, and $p(x) \rightarrow 1$, as $x \rightarrow 0$. Furthermore, $p=1,$ for every $x < 0$.
     
    \item [(A3)] We assume $f_i(\mu_i^R, 0)= r^R(\mu_i^R,0)(1-p(\mu_i^R,0)) - h_ir^S >0$, for each $i\in \mc{N}$, i.e., if no other player reviews any task, then each player $i$ has a positive incentive to review tasks with maximum admission rate $\mu_i^R$. 
\end{itemize}

\begin{remark}
 The rate of return $r^R$  and the constraint probability $p$ can be easily designed to accommodate (A1-A3).
Under Assumptions (A1) and (A2), the incentive function $f_i(\lambda_i^R,\cdot)$ is strictly concave in $\lambda_i^R$, which means for a fixed $\lambda_{-i}^R$, the player $i$ has diminishing marginal incentive to review tasks. 
We make Assumption (A3) to provide positive incentives for players to review tasks with their maximum review admission rate $\mu_i^R$, if no other player chooses to review any task. 
We can design game $\Gamma$ to satisfy Assumption (A3) by ensuring that the following conditions hold:
\begin{enumerate}
    \item  $r^R(\mu_i^R,0) > r^S$, and ${\mu_i^S}\leq \mu_i^R$, for each $i\in \mc{N}$, and 
    \item $\mu_i^R \ll {\mu_T^S}/{a_i}$, {or equivalently $\sum_{j \in \mc{N},\ j \neq i} \mu_j^S \gg \mu_i^R$,} for each $i\in \mc{N}$.
\end{enumerate}

\medskip
{If the latter condition holds, then $x$ is large, and consequently, the constraint probability $p(\mu_i^R,0) \approx 0$, for each $i\in \mc{N}$.} For most practical purposes, servicing a task requires more time than reviewing it, i.e., $\mu_i^S \leq \mu_i^R$. Therefore, condition (i) can be easily satisfied by designing rewards such that $r^R(\mu_i^R,0)> r^S$, for each $i \in \mc{N}$.
If the total service admission rate of all the players except player $i$
is much higher than the maximum review admission rate of player $i$, i.e. $\sum_{j \in \mc{N},\ j \neq i} \mu_j^S \gg \mu_i^R$, for each $i \in \mc{N}$, then condition (ii) holds. Notice that for a large team of agents where a single agent does not have much impact on the overall service rate, condition (ii) is true. We refer the reader to Section~\ref{Numerical Illustrations} for an example. \oprocend
\end{remark}



\begin{theorem}[\bit{Existence of PNE}]\label{thm:thm1}
The CPR game $\Gamma$, under Assumptions (A1-A3), admits a PNE.
\end{theorem}
\begin{proof}
See Appendix~\ref{Existence_appendix} for the proof.
\end{proof}

Let $f'_i(\lambda_i^{R}, \lambda_{-i}^{R})$ be the first partial derivative of $f_i(\lambda_i^{R}, \lambda_{-i}^{R})$ with respect to $\lambda_i^R$. We now provide a corollary that characterizes a PNE of CPR game $\Gamma$.

\begin{corollary}[\bit{PNE}]\label{corollary2}
For the CPR game $\Gamma$, under Assumptions (A1-A3), the following statements hold for a PNE $\lambda^{R^*}= [\lambda_1^{R^*}, \ldots, \lambda_N^{R^*}]$ with $x^*=\mu_T^S-\sum_{i=1}^{N}a_i\lambda_i^{R^*}$:
\begin{enumerate}
    \item $f'_i(\lambda_i^{R^*}, \lambda_{-i}^{R^*}) <0$ (or $\frac{d f_i}{d x}(x^*)>0$) for every player;
    \item $\lambda_i^{R^*}=0$, if and only if, $f_i(\lambda_i^{R^*},\lambda_{-i}^{R^*})\le 0$ ; and
    \item $\lambda_i^{R^*}$ is non-zero and satisfies the following implicit equation if and only if $f_i(\lambda_i^{R^*},\lambda_{-i}^{R^*}) = f_i(x^*)>0$ at PNE, where
     \begin{equation}\label{eq:PNE_min}
       \lambda_i^{R^*}= \min \big\{ \tilde{\lambda}_i^{R^*}, \ \mu_i^R \big\} , 
    \end{equation}
   with $ \tilde{\lambda}_i^{R^*} = -\frac{f_i(\lambda_i^{R^*}, \lambda_{-i}^{R^*})}{f'_i(\lambda_i^{R^*},\lambda_{-i}^{R^*})}= \frac{f_i(x^*)}{a_i\frac{d f_i}{d x}(x^*)}$.
\end{enumerate}
\end{corollary}

\begin{proof}
See Appendix~\ref{Corollary_appendix} for the proof.
\end{proof}

\begin{proposition}[\bit{Structure of PNE}]\label{proposition:Nash_Equilibrium}
For the CPR game $\Gamma$ with players ordered in increasing order of $h_i$,  
let $\lambda^{R^*}= [\lambda_1^{R^*},\lambda_2^{R^*}, \ldots, \lambda_N^{R^*}]$ be a PNE. Then, the following statements hold:
\begin{enumerate}
    \item If, for any player $k_1$, $\lambda_{k_1}^{R^*} < \mu_{k_1}^R$, then $a_{k_1}\lambda_{k_1}^{R^*} \ge a_{k_2}\lambda_{k_2}^{R^*}$ and $\lambda_{k_1}^{R^*} \ge \lambda_{k_2}^{R^*}$, for each $k_2 > k_1$; and
    
    \item {if $\lambda_l^{R^*}=0$, for any $l \in \mc{N}$, then $\lambda_i^R =0$, for each $i \in \{ j \in \mc{N} \ | \ j \ge l \}.$}
\end{enumerate}
\end{proposition}
\begin{proof}
See Appendix~\ref{structure_appendix} for the proof.
\end{proof}


It follows from Proposition~\ref{proposition:Nash_Equilibrium} that 
the review admission rate of a player $i$ at a PNE is monotonically decreasing with the ratio $h_i$. Therefore, at a PNE, as the heterogeneity in terms of $h_i$ among the players becomes very large, players with small (respectively, large) $h_i$ review  tasks  with  high (respectively, zero) review admission rate. { We will show in Lemma~\ref{lemma:social welfare solution} that the PNE shares these characteristics with the social welfare solution, which we define in Section~\ref{Social}. We illustrate this further in Section~\ref{Numerical Illustrations}. }
\medskip

\begin{theorem}[\bit{Uniqueness of PNE}]\label{thm:thm2}
The PNE admitted by the CPR game $\Gamma$, under assumptions (A1-A3), is unique.
\end{theorem}

\begin{proof}
See Appendix~\ref{Uniqueness_appendix} for the proof.
\end{proof}

\section{Convergence to the Nash Equilibrium} \label{Convergence}

{We now show that the proposed CPR game $\Gamma$ under Assumptions (A1-A3) belong to the class of \textit{Quasi Aggregative games}~\cite{jensen2010aggregative} as defined below.}

\medskip
 \begin{definition}[\bit{Quasi Aggregative game}] \label{definition3}
 {Consider a set of players $\mc{N}$, where each player $i \in \mc{N}$ has a strategy set $S_i$, and a utility function $u_i$. Let $S=\prod_{i \in \mc{N}}S_i$ be the joint strategy space of all the players, and $S_{-i}=\prod_{j \in \mc{N},j \neq i}S_j$ be the joint strategy space of all the players except player $i$. A game $\Gamma=(\mc{N},\{S_i\}_{i \in \mc{N}},\{u_i\}_{i \in \mc{N}})$ is a quasi-aggregative game with aggregator $g: S \mapsto \mathbb{R}$, if there exists continuous functions $F_i: \mathbb{R}\times S_i \mapsto \mathbb{R}$ (the shift functions) and $\sigma_i: S_{-i} \mapsto X_{-i} \subseteq \mathbb{R}, \ i \in \mc{N}$ (the interaction functions) such that the utility functions $u_i$ for each player $i \in \mc{N}$ can be written as:}
 \begin{equation}{\label{quasi}}
     u_i(s)=\tilde u_i(\sigma_i(s_{-i}),s_i),
 \end{equation}
 where $\tilde u_i:X_{-i}\times S_i \mapsto \mathbb{R}$, and
 \begin{equation}
     g(s)=F_i(\sigma_i(s_{-i}),s_i), \ \text{for all} \ s \in S \ \text{and} \ i \in \mc{N}.
\end{equation}
 An alternative, but less general way of defining a quasi-aggregative game replaces~\eqref{quasi} in the definition with:
 \begin{equation}
      u_i(s)=\overline{u}_i(g(s),s_i),
 \end{equation}
 where $\overline{u}_i:X\times S_i \mapsto \mathbb{R}$, and $X= \{g(s) \ | s \in S \}\subseteq \mathbb{R}$.\\
 \end{definition}
 For the CPR game $\Gamma$, let $\sigma_i(\lambda_{-i}^R)=\sum_{j=1, j \neq i }^{N}{a_j}\lambda_j^R$ 
 and $g(\lambda^R)=F_i(\sigma_i(\lambda_{-i}^R),\lambda_i^R)=\sum_{j=1, j \neq i }^{N}{a_j}\lambda_j^R + a_i\lambda_i^R$ be the interaction functions and shift functions, respectively.
 The expected utility $\tilde{u}_i$, which is defined in~\eqref{eq:7}, can be re-written in the form
 \begin{equation}
     \tilde{u}_i(\lambda_i^R,\lambda_{-i}^R)= \tilde{u}_i(\sigma_i(\lambda_{-i}^R),\lambda_i^R).
 \end{equation}
 Hence, the CPR game $\Gamma$ is a quasi-aggregative game.

\medskip
Specializing~\cite[Theorem 1]{jensen2010aggregative} to the CPR game $\Gamma$, we obtain that if the best response for all the players is non-increasing in the interaction function $\sigma_i(\lambda_{-i}^R) =\sum_{j=1, j \neq i }^{N}{a_j}\lambda_j^R$, the CPR game $\Gamma$ is a best response pseudo-potential game~\cite{schipper2004pseudo} as defined below. 

\begin{definition}[\bit{Best response (pseudo)-potential game}] \label{best_response_pseudo_ponential}
A game $\Gamma=(\mc{N},\{S_i\}_{i \in \mc{N}},\{\tilde u_i\}_{i \in \mc{N}})$ is a best response pseudo-potential game if there exists a continuous function $\phi : S \mapsto \mathbb{R}$ such that for every $i \in \mc{N}$, 
\[
b_i(\lambda_{-i}^R) \supseteq \argmax_{\lambda_i^R \in S_i} \; \phi(\lambda_i^R, \ \lambda_{-i}^R), 
\]
where $b_i(\lambda_{-i}^R)$ is the best response of player $i$ to the review admission of other players $\lambda_{-i}^R$. Furthermore, if 

\[
b_i(\lambda_{-i}^R) = \argmax_{\lambda_i^R \in S_i} \; \phi(\lambda_i^R, \ \lambda_{-i}^R), 
\]
then the game $\Gamma$ is a best response potential game.
\end{definition}

We now establish that the best response for each player is non-increasing in $\sigma_i$. 

\begin{lemma}[\bit{Non-increasing best response}]\label{lemma:best_response_non_increasing}
 For the CPR game $\Gamma$, under Assumptions (A1-A2), the best response mapping $b_i(\lambda_{-i}^R)$ is non-increasing in $\sigma_i(\lambda_{-i}^R)$, for each $i \in \mc{N}$, where $\sigma_i(\lambda_{-i}^R)=\sum_{j=1, j \neq i }^{N}{a_j}\lambda_j^R$.
\end{lemma}

\begin{proof}
See Appendix~\ref{non_increasing_appendix} for the proof.
\end{proof}

 Furthermore, Remark 1 in~\cite{dubey2006strategic} states that a best response pseudo-potential game with a unique best response, is an instance of best response potential game~\cite{voorneveld2000best}. Therefore, the CPR game $\Gamma$, with its unique (Lemma~\ref{lemma:best_response_mapping}) and non-increasing best response $b_i$ in $\sigma_i(\lambda_{-i}^R)$ (Lemma~\ref{lemma:best_response_non_increasing}), is a best response potential game. Hence, simple best response dynamics such as sequential best response dynamics \cite{dubey2006strategic} and simultaneous best response dynamics \cite{jensen2009stability} converge to the unique PNE.

\section{Social Welfare and Inefficiency of PNE}\label{Social}

In this section, we characterize the social welfare solution and provide analytic upper bounds on inefficiency measures for the PNE.
\subsection{Social Welfare}

Social welfare corresponds to the optimal (centralized) allocation by players with respect to a social welfare function. To characterize the effect of self-interested optimization of each agent, we compare the decentralized solution (PNE of the CPR game) with the centralized optimal solution (social welfare). 

We choose a typical social welfare function $\Psi(\lambda^R) : S \mapsto \mathbb{R}$ defined by the sum of expected utility of all players, i.e.,
\begin{align}
  \Psi &= \sum_{i=1}^{N}\tilde u_i = \sum_{i=1}^{N}[\mu_i^Sr^S +\lambda_i^Rf_i(x)]  \nonumber \\
        &=\mu_T^Sr^S + \lambda_T^Rr^R(x)(1-p(x)) - r^S\sum_{i=1}^{N}h_i\lambda_i^R \nonumber \\
      & = (\lambda_T^R+x) r^S + \lambda_T^Rr^R(x)(1-p\left(x)\right). \label{eq:8}
 \end{align}

A \textit{social welfare solution} is an optimal allocation that maximizes the social welfare function. 

\medskip
\begin{lemma}[\bit{Social welfare solution}]\label{lemma:social welfare solution}
For the CPR game $\Gamma$ with constraint $\sum_{i=1}^N a_i \lambda_i^R =c$ , for any given $c \in \mathbb{R}_{\ge 0}$, and players ordered in increasing order of $h_i$, the associated social welfare solution, $\lambda^R \in S$ is given by:
 \begin{equation*}
    \lambda^R= \left[\mu_1^R, \ \mu_2^R ,\ldots,  \ \mu_{k-1}^R, \ \frac{1}{a_k}(c -\sum_{i=1}^{k-1}a_i\mu_i^R),\  0, \ldots , \ 0 \right],
\end{equation*}
where $k$ is the smallest index such that $\sum_{i=1}^{k-1} a_i\mu_i^R \le c< \sum_{i=1}^{k} a_i\mu_i^R$.
Furthermore, since  $\sum_{i=1}^N a_i \lambda_i^R \in [0, \ \mu_T^S + \mu_T^R]$, a bisection algorithm can be employed to compute optimal $c$ and hence, the optimal social welfare solution.
\end{lemma}
\medskip

\begin{proof}
 { Under the constraint $\sum_{i=1}^N a_i \lambda_i^R =c$ (equivalently, $x= \mu_T^S -c$), $\Psi$ is a strictly increasing function of  $\lambda_T^R$.
     Therefore, for a fixed $\sum_{i=1}^N a_i \lambda_i^R =c$, $\lambda_T^R$ is maximized by selecting $k\!-\!1$ players with smallest $a_i$'s (equivalently, $h_i$) to operate at their highest review admission rate, where the value of $k$ is selected such that $\sum_{i=1}^{k-1} a_i\mu_i^R \le c< \sum_{i=1}^{k} a_i\mu_i^R$.   
     Finally, the $k$-th player in the ordered sequence is selected to operate at a review admission rate such that the constraint $\sum_{i=1}^N a_i\lambda_i^R =\sum_{i=1}^{k} a_i\lambda_i^R=c$, is satisfied. Therefore, the social welfare solution is of the form,
     
     \begin{equation*}
    \lambda^R= \left[\mu_1^R, \ \mu_2^R ,\ldots,  \ \mu_{k-1}^R, \ \frac{1}{a_k}(c -\sum_{i=1}^{k-1}a_i\mu_i^R) ,\  0, \ldots , \ 0 \right].
    \end{equation*}
     
Furthermore, for the function $r^R(x)$ and $p(x)$ satisfying Assumptions (A1-A2), $\Psi$ is strictly concave in $x$, i.e., $\frac{\partial^2\Psi}{\partial x^2}=\lambda_T^R\frac{d^2f_i}{dx^2}<0$ (Lemma~\ref{lemma:incentive_function}). With the known form of the social welfare solution, the value of $c$, which corresponds to the unique maximizer $x$ of $\Psi$, can be computed efficiently by employing a bisection algorithm~\cite{burden19852}.}
\end{proof}
\subsection{Inefficiency of the PNE}
We consider three measures of the inefficiency for the PNE: a) Price of Anarchy (PoA), b) Ratio of total review admission rate ($\eta_{TRI}$), and c) Ratio of Latency ($\eta_{LI}$), which are described by 

 \resizebox{.95\linewidth}{!}{
   \begin{minipage}{\linewidth}
\begin{equation*}
  \!\!\!    PoA = \frac{(\Psi)_{SW}}{(\Psi)_{PNE}}, \ \eta_{TRI}=\frac{({\lambda_T^R})_{SW}}{{(\lambda_T^R})_{PNE}}, \ \eta_{LI}=\frac{(\sum_{i=1}^{N}a_i{\lambda_i^R})_{PNE}}{(\sum_{i=1}^{N}a_i{\lambda_i^R})_{SW}},
\end{equation*}
     \end{minipage}
     }
respectively. While PoA is a widely used measure of the inefficiency, $\eta_{TRI}$ and $\eta_{LI}$ capture the inefficiency of the PNE based on the total review admission rate and the latency (inverse of throughput), respectively. Since incentivizing team collaboration is of interest, all three measures capture the inefficiency of the PNE well.



{We now provide an analytic upper bound for each of these measures of inefficiency for the PNE. To this end, we assume that $\min_i\{\mu_i^S\} > \frac{\mu_T^S h_N}{N(1+h_N)}$. For scenarios wherein 
servicing a task requires much more time than reviewing it, i.e., $\mu_N^S \ll \mu_N^R$ ($h_N \rightarrow 0$), the assumption reduces to $\min_i\{\mu_i^S\} >0$.}
\begin{theorem}[\bit{Analytic bounds on PNE inefficiency}]\label{thm:thm3}
For the CPR game $\Gamma$, under assumptions (A1-A3), and {$\min_i\{\mu_i^S\} > \frac{\mu_T^S h_N}{N(1+h_N)}$}, the inefficiency metrics for the PNE are upper bounded by
\begin{equation}
  PoA < \frac{\mu_T^S a_N}{\mu_T^S - \overline{x}}, \ \eta_{TRI}< \frac{\mu_T^S a_N}{(\mu_T^S - \overline{x})a_1}, \  \eta_{LI}< \frac{\mu_T^S}{\mu_T^S - \overline{x}},
  \end{equation}
%
where $\overline{x}$ is the unique maximizer of $f_i$, i.e., $\frac{d f_i}{d x} (\overline{x})=0$ .
\end{theorem}

\begin{proof}
See Appendix~\ref{bound_appendix} for the proof.
\end{proof}

\noindent 
\textbf{Example 1:} We show analytic upper bounds on inefficiency measures for the PNE for a specific class of exponential functions $r^R(x)=A[1-\exp\{B(x-\mu_T^S)\}]$ and $p(x)= \exp(-Bx)$, 
where $A$ and $B$ are positive constants, and $x \in [0, \ \mu_T^S]$.

Setting $\frac{d f_i}{d x} (\overline{x})= 0 $, we obtain $\overline{x}=\frac{\mu_T^S}{2}$. Using Theorem~\ref{thm:thm3}, we get PoA $< 2a_N, \eta_{TRI}< \frac{2a_N}{a_1},$ and $\eta_{LI} <2$. {For $\mu_N^S \ll \mu_N^R$,} PoA $< 2a_N \rightarrow 2$, and $\eta_{TRI} < \frac{2a_N}{a_1} < 2a_N \rightarrow 2$.

\section{Numerical Illustrations}\label{Numerical Illustrations}

In this section, we present numerical examples illustrating the uniqueness of PNE and the variation of inefficiency 
with increasing heterogeneity among the players. 



In our numerical illustrations, we obtain the PNE by simulating the sequential best response dynamics of players with randomized initialization of their strategy. We verify the uniqueness of the PNE for different choices of functions, $r^R(x)$ and $p(x)$ satisfying Assumptions (A1-A2), and by following sequential best response dynamics with multiple random initializations for the strategy of each player. Furthermore, in our numerical simulations,  we relax Assumption (A3) and still obtain a unique PNE.

\begin{figure}
	\centering
\begin{subfigure}[b]{0.48\linewidth}
	    \centering
        \includegraphics[width=1\linewidth, height=1\linewidth, keepaspectratio]{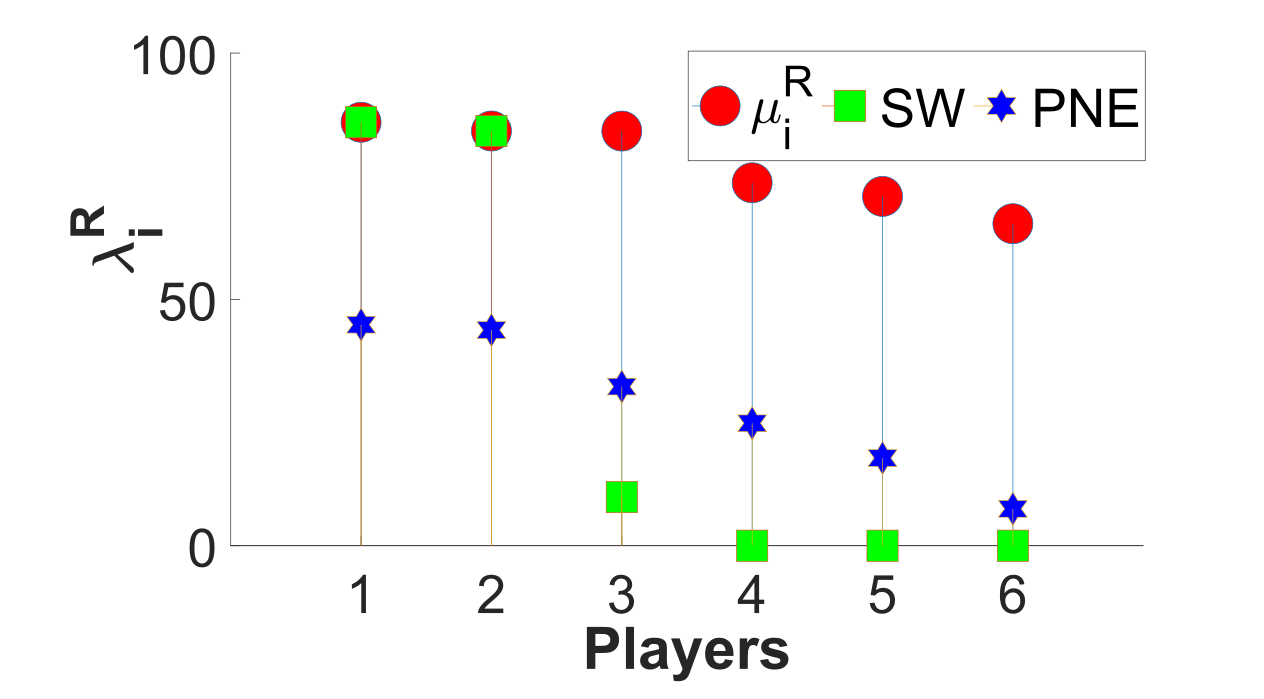}
        \caption{}
        \label{fig:p1}
    \end{subfigure}
    ~
\begin{subfigure}[b]{0.48\linewidth}
	    \centering
        \includegraphics[width=1\linewidth, height=1\linewidth, keepaspectratio]{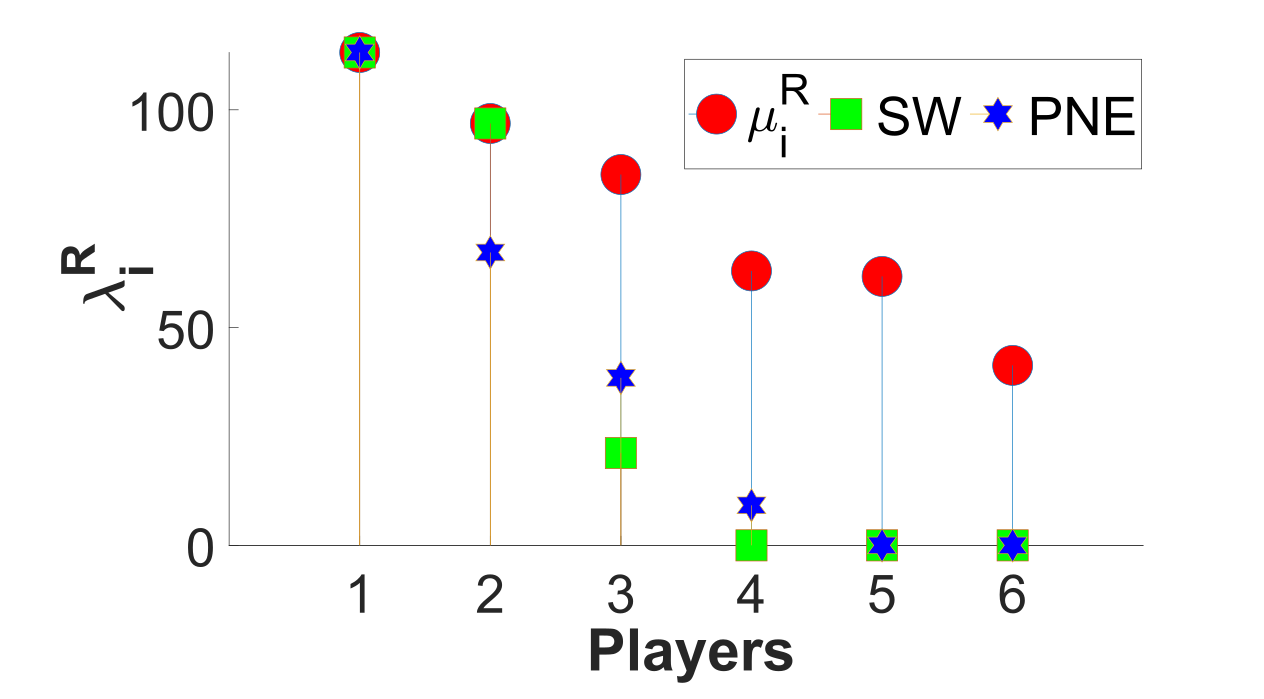}
        \caption{}
        \label{fig:p2}
    \end{subfigure}

    \caption{\footnotesize Social welfare solution (SW) and pure Nash equilibrium for a) low and b) high heterogeneity among players, respectively.  Red circles show the maximum review admission rate ($\mu_i^R$) for player $i$.}
    \label{fig:Social_welfare}
\end{figure}

\begin{figure*}
\centering
	\begin{subfigure}[b]{0.30\textwidth}
	    \centering
	    \includegraphics[width=1.15\linewidth, height=1.15\linewidth, keepaspectratio]{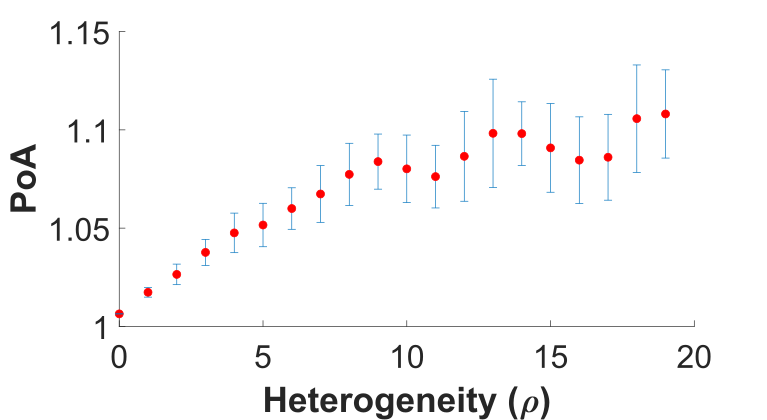}
        \caption{}
        \label{fig:Picture4}
    \end{subfigure}
    ~~
    \begin{subfigure}[b]{0.30\textwidth}
	    \centering
	    \includegraphics[width=1.15\linewidth, height=1.15\linewidth, keepaspectratio]{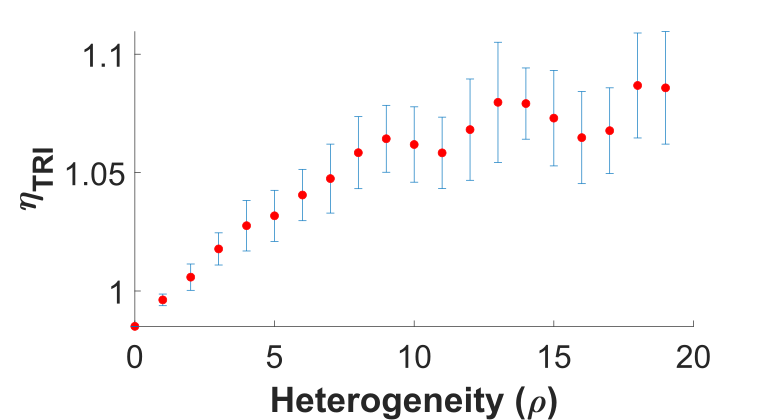}
        \caption{}
        \label{fig:Picture40}
    \end{subfigure}
    ~~
    \begin{subfigure}[b]{0.30\textwidth}
	    \centering
	    \includegraphics[width=1.05\linewidth, height=1.05\linewidth, keepaspectratio]{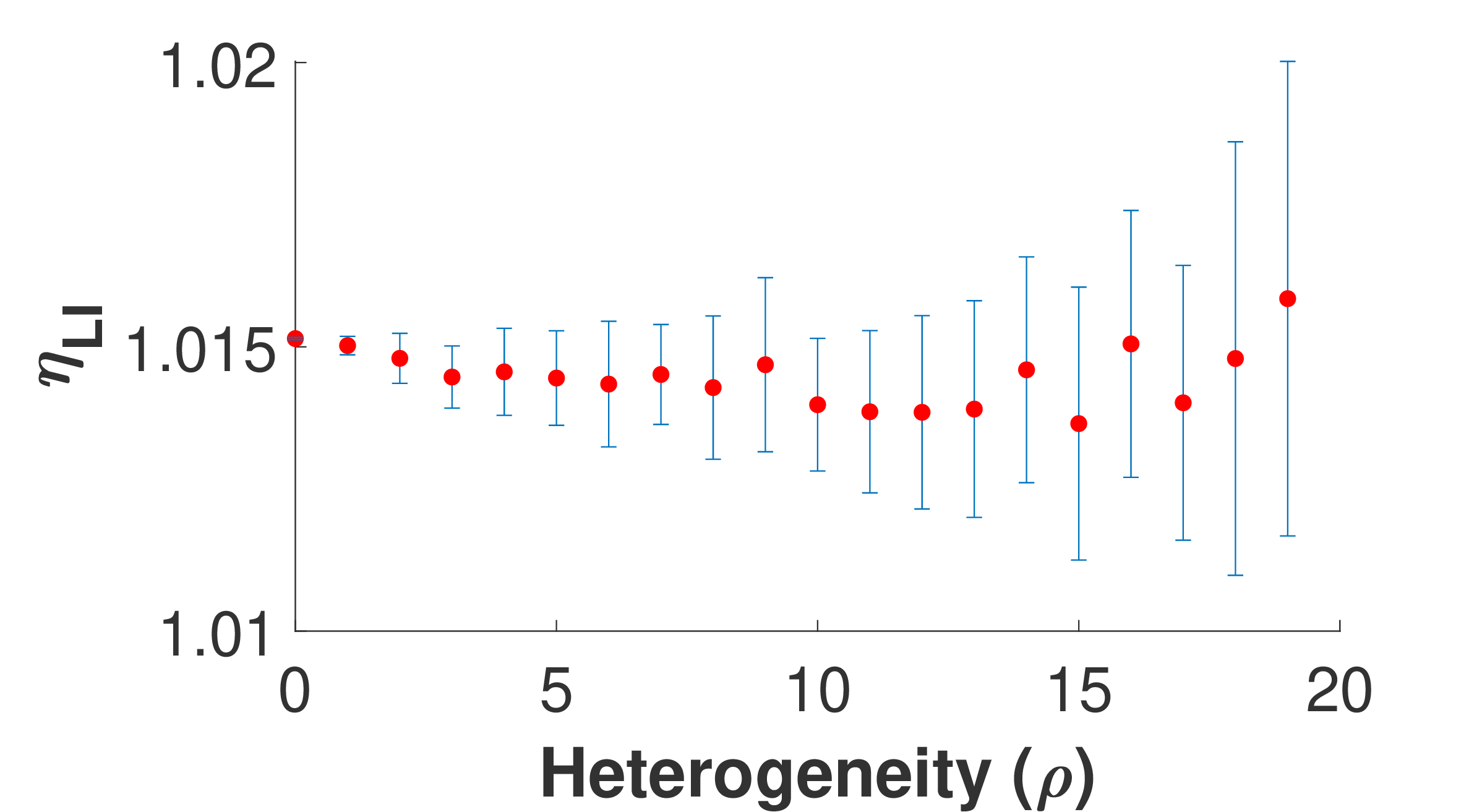}
        \caption{}
        \label{fig:Picture41}
    \end{subfigure}
    \begin{subfigure}[b]{0.30\textwidth}
	    \centering
	    \includegraphics[width=1.15\linewidth, height=1.15\linewidth, keepaspectratio]{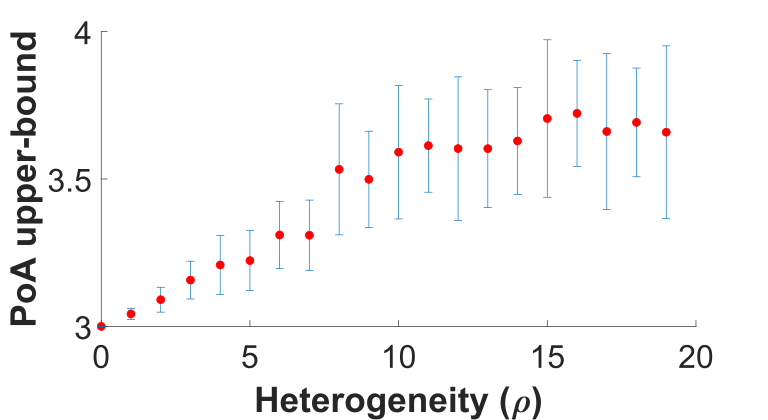}
        \caption{}
        \label{fig:Picture49}
    \end{subfigure}
    ~~
    \begin{subfigure}[b]{0.30\textwidth}
	    \centering
	    \includegraphics[width=1.15\linewidth, height=1.15\linewidth, keepaspectratio]{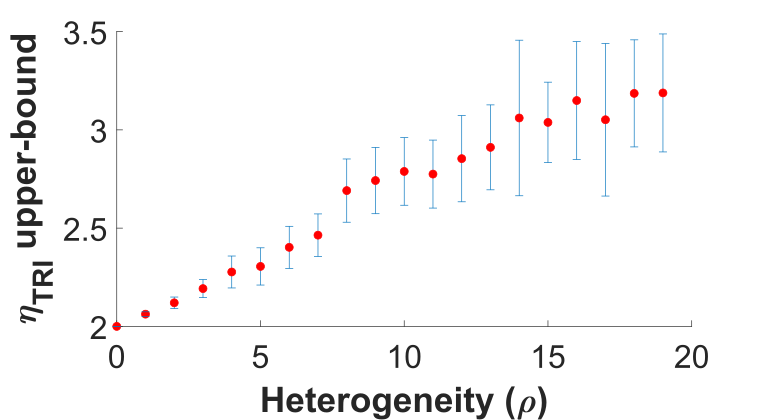}
        \caption{}
        \label{fig:Picture50}
    \end{subfigure}
    ~~
    \begin{subfigure}[b]{0.30\textwidth}
	    \centering
	    \includegraphics[width=1.05\linewidth, height=1.05\linewidth, keepaspectratio]{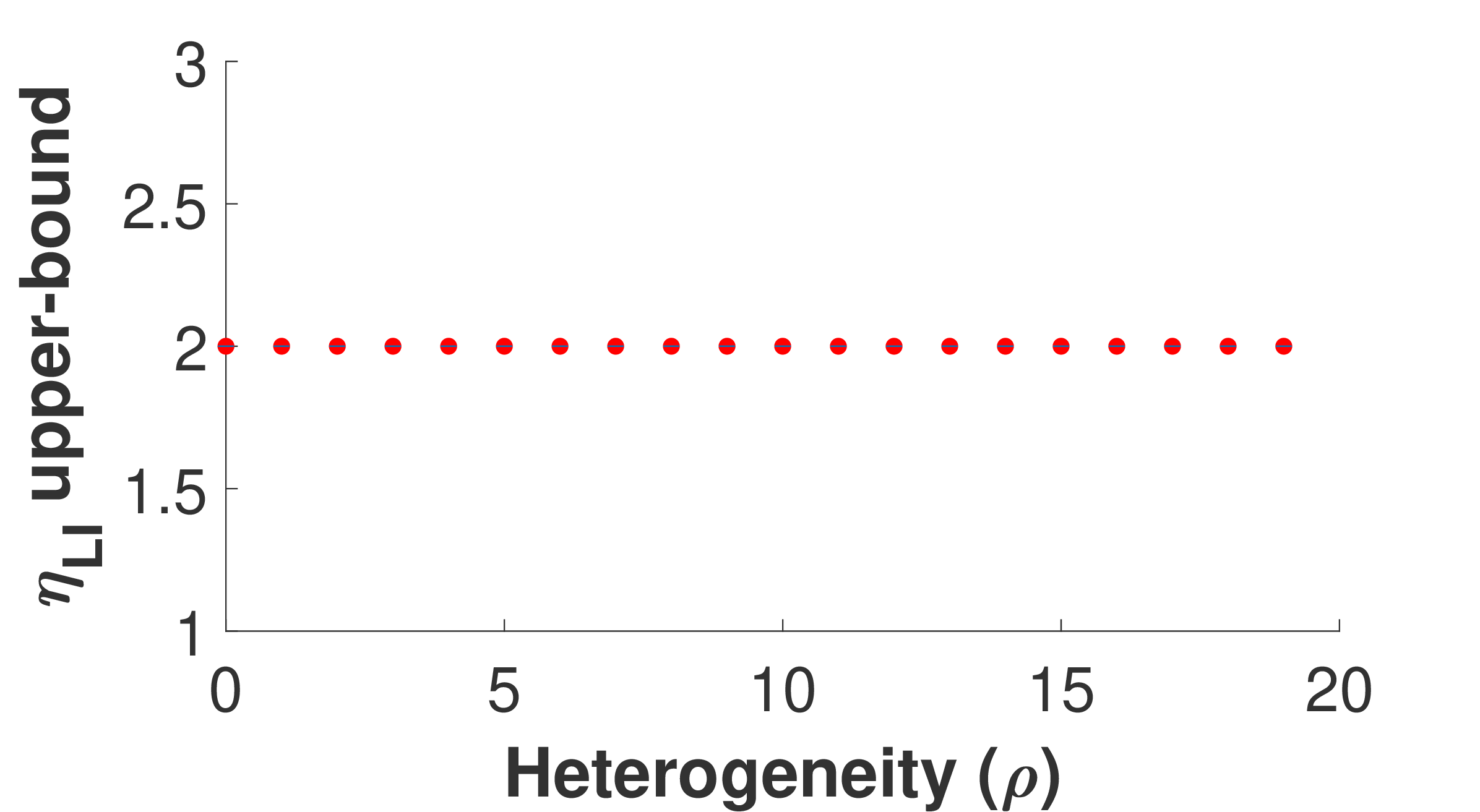}
        \caption{}
        \label{fig:Picture51}
        \end{subfigure}
    \caption{\footnotesize Empirical a) PoA, b) $\eta_{TRI}$, and c) $\eta_{LI}$, along with analytic upper bounds for d) PoA, e) $\eta_{TRI}$, and f) $\eta_{LI}$ with increasing heterogeneity ($\rho$) among the agents.}
    \label{fig:Inefficiencies}
\end{figure*}

An example illustration is shown in Fig.~\ref{fig:Social_welfare}, where we show the social welfare solution (obtained using fmincon in MATLAB) and PNE for low and high heterogeneity in terms of variation in $h_i$ among players, respectively. For our numerical illustrations, we choose the number of players, $N=6$, and choose the functions $r^R(x)$ and $p(x)$, satisfying Assumptions (A1-A2) as following:
\begin{equation*}{\label{r_choice}}
    r^R(\lambda_i^R, \lambda_{-i}^R) = r^R(x)=5[1-\exp\{0.5(x-\mu_T^S)\}],
\end{equation*}
\begin{equation*}{\label{p_choice}}
    p_i^R(\lambda_i^R, \lambda_{-i}^R) = p(x) =\begin{cases}
    1, &  \text{if} \ \ \ \ x\le 0,\\ 
    \exp(-0.5x), & \text{otherwise},
    \end{cases}
\end{equation*}
where $x=\mu_T^S-\sum_{i=1}^{N}a_i\lambda_i^R$ is the slackness parameter. To characterize the heterogeneity among the players, we sample the player's maximum service admission rate $\mu_i^S$ and maximum review admission rate $\mu_i^R$ at random from normal distributions with fixed means,  $M_{\mu_S} \in \mathbb{R}_{>0}$, and $ M_{\mu_R} \in \mathbb{R}_{>0}$, and identical standard deviation, $\rho \in \mathbb{R}_{>0}$. {We only consider realizations that satisfy $\mu_i^S \le \mu_i^R$ for all the players, and hence, $h_i \le 1$. For most practical purposes, where servicing a task requires much more time than reviewing it, the assumption $\mu_i^S \le \mu_i^R$ holds true.} Any non-positive realizations were discarded. We consider the standard deviation of the distributions as the measure of heterogeneity among the players.

    
    


 Fig.~\ref{fig:Social_welfare} shows that in the social welfare solution, players with low ratio of $h_i$ review the tasks at maximum review admission rate and players with high ratio of $h_i$ drop out of the game. At PNE, the strategy profile of players follow the characteristics described by Proposition~\ref{proposition:Nash_Equilibrium}. Lastly, with the increase in heterogeneity among the players, the PNE starts to approach the social welfare solution. 

{Fig.~\ref{fig:Picture4}-\ref{fig:Picture41} and Fig.~\ref{fig:Picture49}-\ref{fig:Picture51} shows the variation of different measures of inefficiency for PNE, and their corresponding analytic upper bounds (see Theorem~\ref{thm:thm3}), with increasing heterogeneity among the players.} Fig.~\ref{fig:Picture4} shows the plot of PoA with increasing heterogeneity. In case of homogeneous players, i.e., $\rho=0$, we obtain $PoA=1$, which we establish in Lemma~\ref{lemma:homogeneous PoA}. As we initially increase the heterogeneity among the players, PNE starts to deviate from the social welfare solution, resulting in an increase in the PoA.
We note that PoA $\le 1.15$, suggesting that the unique PNE is close to the optimal centralized social welfare solution. Fig.~\ref{fig:Picture40} and~\ref{fig:Picture41} shows $\eta_{TRI}$ and $\eta_{LI}$, which are other relevant measures of inefficiency for our problem. 
It is evident from Fig.~\ref{fig:Inefficiencies}, that all three measures of inefficiency are close to $1$, therefore suggesting near-optimal PNE solution.

\section{Conclusions and Future Directions} \label{Conclusions}
We studied incentive design mechanisms to facilitate collaboration in a team of heterogeneous agents that is collectively responsible for servicing and subsequently reviewing a stream of homogeneous tasks. The heterogeneity among the agents is based on their skill-sets and is characterized by their mean service time and mean review time. To incentivize collaboration in the heterogeneous team, we designed a Common-Pool Resource (CPR) game with appropriate utilities and showed the existence of a unique PNE. We showed that the proposed CPR game is an instance of the best response potential game and by playing the sequential best response against each other, players converge to the unique PNE. We characterized the structure of the PNE and showed that at the PNE, the review admission rate of the players decreases with the increasing ratio of {$h_i=\frac{\mu_i^S}{\mu_i^R}$}, i.e., the review admission rate is higher for the players that are ``better" at reviewing the tasks than servicing the tasks (characterized by their average service and review time). Furthermore, we consider three different inefficiency metrics for the PNE, including the Price of Anarchy (PoA), and provide an analytic upper bound for each metric. Additionally, we provide numerical evidence of their proximity to  unity, i.e., the unique PNE is close to the optimal centralized social welfare solution.

There  are  several  possible  avenues  of  future  research. It is of interest to extend the results for a broader class of games with less restrictive choice of utility functions, i.e., games that are not quasi-aggregative or commonly used games of weak strategic substitutes (WSTS)~\cite{dubey2006strategic} or complements (WSTC)~\cite{dubey2006strategic}. An interesting open problem is  to consider  a  team  of  agents  processing  stream of heterogeneous tasks. In such a setting, incentivizing team collaboration based on the task-dependent skill-set of the agents is also of interest.

\normalsize
\appendix
\subsection{Proof of Theorem 1 [Existence of PNE]}\label{Existence_appendix}

We prove Theorem~\ref{thm:thm1} using Brouwer's fixed point theorem~\cite[Appendix C]{basar1999dynamic} applied to the best response mapping with the help of following lemmas (Lemmas~\ref{lemma:incentive_function}-\ref{lemma:best_response_continuous}). 
Recall that $b_i(\lambda_{-i}^R)$ is the best response of player $i$ to the review admission rates of other players $\lambda_{-i}^R$. For brevity of notation, we will 
represent $r^R(\lambda_i^R, \lambda_{-i}^R), \  p(\lambda_i^R, \lambda_{-i}^R), \ f_i(\lambda_i^R, \lambda_{-i}^R), \ \tilde u_i(\lambda_i^R, \lambda_{-i}^R) $ using $r^R, \ p, \ f_i, \ \tilde u_i$, respectively. Furthermore, let $q'$ and $q''$, respectively, represent the first and the second partial derivatives of a generic function $q$ with respect to $\lambda_i^R$. 
\medskip
\begin{lemma}[\bit{Strict concavity of incentive}]\label{lemma:incentive_function}
For the CPR game $\Gamma$, under Assumptions (A1-A2), the incentive function $f_i: S \mapsto \mathbb{R}$ is strictly concave in $\lambda_i^R$, for  $\lambda_i^R \in [0, \ \overline{\lambda}_i^R]$ and any fixed $\lambda_{-i}^R$. Equivalently, $f_i(x)$ is strictly concave in $x$ for $x \in [0, \ \mu_T^S-\sum_{j \in \mc{N}, j \neq i}a_j\lambda_j^R].$
\end{lemma}

\begin{proof}
{Recall from~\eqref{eq:def-fi} that
\begin{equation*}
    f_i(\lambda_i^R,\lambda_{-i}^R)=f_i(x)=r^R(x)(1-p(x)) -h_ir^S.
\end{equation*}}
The first and the second partial derivative of the incentive function $f_i$ with respect to $\lambda_i^R$ in the interval $\lambda_i^R \in [0, \ \overline{\lambda}_i^R]$ are given by:
\begin{subequations}{\label{eq:12}}
\begin{equation}{\label{eq:first derivative}}
f_i'=(r^R)'(1-p)- r^Rp'= -a_i\frac{df_i}{dx},
\end{equation}
\begin{equation}{\label{eq:second derivative}}
  f_i'' =(r^R)''(1-p)- 2(r^R)'p' -r^Rp''= a_i^2\frac{d^2 f_i}{dx^2}.
\end{equation}
\end{subequations}

From Assumptions (A1) and (A2), we have $f_i''<0$ and $\frac{d^2f_i}{dx^2}<0$ in the interval where derivative of $f_i$ exists, thereby proving the strict concavity of $f_i$ in $\lambda_i^R$ and $x$. 
\end{proof}
\medskip
\begin{lemma}[\bit{Best response mapping}]\label{lemma:best_response_mapping}
For the CPR game $\Gamma$, under Assumptions (A1-A2), the best response mapping $b_i(\lambda_{-i}^R)$ is unique for any $\lambda_{-i}^R \in S_{-i}$ and is given by:    
\resizebox{0.9\linewidth}{!}{
    \begin{minipage}{\linewidth}
\begin{equation*}{\label{eq:10}}
    b_i(\lambda_{-i}^R) = \begin{cases}
    0, & \text{if }  f_i(\lambda_i^R,\cdot) \leq 0, \ \ \forall \lambda_i^R \in S_i,\\
    \alpha_i , & \text{if } \exists \alpha_i \in S_i \ \text{s.t. } \ \frac{\partial \tilde u_i}{\partial\lambda_i^R}(\alpha_i)=0, \ \text{and } f_i(\alpha_i,\cdot)>0, \\
    \mu_i^R, & \text{otherwise }.\\
    \end{cases}
\end{equation*}
\end{minipage}
}
\end{lemma}
\medskip
\begin{proof}
We establish uniqueness of the best response mapping through the following three cases.

\medskip
\noindent \textbf{Case 1:} $f_i(\lambda_i^R, \cdot) \leq 0,$ for every $\lambda_i^R \in S_i$.

\medskip
If for a given $\lambda_{-i}^R \in S_{-i}$, $f_i(\lambda_i^R, \cdot) \leq 0,$ for every $\lambda_i^R \in S_i$, then from~\eqref{eq:7}, $\tilde u_i(\lambda_i^R, \ \lambda_{-i}^R)$ admits a unique maximum at $\lambda_i^R=0$, and  therefore, $b_i(\lambda_{-i}^R)=0$ is the unique best response.


\medskip
\noindent \textbf{Case 2:} There exists a non-empty interval $ \overline{S_i} \subset S_i$, such that $f_i(\lambda_i^R, \cdot) > 0$, and $f_i'(\lambda_i^R, \cdot) <0,$ for every $\lambda_i^R \in \overline{S_i}$.



\medskip
{For any given $\lambda_{-i}^R \in S_{-i}$, recall that the system constraint~\eqref{eq:3} is violated for every $\lambda_i^R \in (\overline{\lambda}_i^R, \mu_i^R] \subset S_i$, and $p(\lambda_i^R, \lambda_{-i}^R)=1$.
Therefore, for every $\lambda_i^R \in (\overline{\lambda}_i^R, \mu_i^R]$, we have}
\begin{equation}{\label{eq:11}}
    f_i= -h_ir^S < 0. 
\end{equation}
Therefore, $b_i(\lambda_{-i}^R) \in [0, \ \overline{\lambda}_i^R] \subset S_i$, for any given $\lambda_{-i}^R \in S_{-i}$. Furthermore, for a fixed $\lambda_{-i}^R$, since $p$ is continuously differentiable with respect to $ \lambda_i^R$, for each $\lambda_i^R \in (0, \ \overline{\lambda}_i^R), \ \tilde u_i$ is a smooth function on the set $[0, \ \overline{\lambda}_i^R]\times S_{-i}$. Hence, the best response, which is a global maximizer of $\tilde u_i$ on the interval $\lambda_i^R \in S_i$, either occurs at the boundary of $S_i$ or satisfies the first order condition, $\frac{\partial \tilde u_i}{\partial \lambda_i^R}(b_i) =0$ (see~\cite{luenberger1984linear}).

Let {there exist} $\alpha_i \in S_i $ such that  
\begin{subequations}\label{eq:alpha-exist}
\begin{equation}
f_i(\alpha_i, \cdot)>0, \quad \text{and} \quad 
\end{equation}
\begin{equation}\label{eq:13}
      \frac{\partial \tilde u_i}{\partial \lambda_i^R}(\alpha_i) =\alpha_if'_i(\alpha_i, \cdot) +   f_i(\alpha_i, \cdot) = 0. 
\end{equation}
\end{subequations}
%
Since $f_i(\alpha_i, \cdot)>0$ and $\alpha_i >0$, ~\eqref{eq:13} has a solution only if $f'_i(\alpha_i, \cdot)<0$. Furthermore, $f_i(\alpha_i, \cdot)>0$ implies $\alpha_i \in [0, \ \overline{\lambda}_i^R]$ (see~\eqref{eq:11}). Therefore, existence of $\alpha_i$ satisfying~\eqref{eq:alpha-exist} implies
there exists a non-empty set $\overline{S_i} \subset [0, \ \overline{\lambda}_i^R] \subset S_i$, such that for each $\alpha_i \in \overline{S_i}$,  $f_i(\alpha_i, \cdot)>0$ and $f'_i(\alpha_i, \cdot)<0$. For any $\lambda_i^R \in \overline{S_i}$, such that $f_i(\lambda_i^R, \cdot)>0$ and $f'_i(\lambda_i^R, \cdot)<0$, using Lemma~\ref{lemma:incentive_function}, we get:
\begin{equation}{\label{eq:14}}
  \frac{\partial^2 \tilde u_i}{\partial {\lambda_i^R}^2} =\lambda_i^Rf''_i +   2f'_i <0.
\end{equation}
Hence, for $\lambda_i^R \in \overline{S_i}$, the expected utility $\tilde u_i$ is strictly concave with a unique global maximizer $\alpha_i \in \overline{S_i}$ {that satisfies $\alpha_i= \min\{-\frac{f_i(b_i, \cdot)}{f'_i(b_i,\cdot)}, \ \mu_i^R\}$} (see~\eqref{eq:13}). 

\medskip
\noindent \textbf{Case 3:} There exists a non-empty interval $\tilde{S_i} \subset S_i$, such that $f_i(\lambda_i^R, \cdot) > 0, $ for every $ \lambda_i^R \in \tilde{S_i}$, and $f_i'(\lambda_i^R, \cdot) \geq 0,$  for any $  \lambda_i^R \in S_i$.

\medskip

Finally, consider the case that $f'_i(\lambda_i^R, \cdot) \ge 0,$ for every $\lambda_i^R \in S_i$, and there exists an interval $\tilde{S_i} \subset S_i$ where $f_i(\lambda_i^R, \cdot) > 0,$ for any $ \lambda_i^R \in \tilde{S_i}$. Since $f'_i(\lambda_i^R, \cdot)\ge0, $ for every $ \lambda_i^R \in S_i $, i.e., $f_i(\lambda_i^R, \cdot)$ is increasing in $\lambda_i^R$,  and therefore, $f_i(\lambda_i^R, \cdot)$ is maximized at $\lambda_i^R=\mu_i^R$. Since there exists a non-empty interval $\tilde{S_i}$ such that $f_i(\lambda_i^R, \cdot) > 0, $ for every $ \lambda_i^R \in \tilde{S_i}$, monotonically increasing $f_i(\lambda_i^R, \cdot)$, it follows $\mu_i^R \in \tilde{S_i}$,  and $f_i(\mu_i^R, \cdot)>0$. Therefore, in the interval $\lambda_i^R \in \tilde{S_i}$, ~\eqref{eq:13} has no solution and the expected utility of player $i$ is strictly increasing in $\lambda_i^R$, i.e., $\frac{\partial{\tilde{u}_i}}{\partial{\lambda_i^R}}>0$, for every $\lambda_i^R \in S_i$. Therefore, the best response is the unique maximum of $\tilde u_i$ which occurs at the boundary $\mu_i^R$.
\end{proof}
\medskip
{We state some important intermediate results from three cases of Lemma~\ref{lemma:best_response_mapping} as a corollary for later discussions.}
\begin{corollary}[\bit{Best response and incentive}]\label{corollary1}
For the CPR game $\Gamma$, under Assumptions (A1-A3), the following statements hold:
\begin{enumerate}
    \item {$b_i=0$, if and only if, $f_i(\lambda_i^R,\cdot)\le 0$, for every $\lambda_i^R \in S_i$; furthermore, $f_i(\lambda_i^R,\cdot)\le 0$, for every $\lambda_i^R \in S_i$ implies $f'_i(\lambda_i^R,\cdot)< 0$, for every $\lambda_i^R \in S_i$};
    \item if there exists an interval $\overline{S_i} \subset S_i$, such that $f_i(\lambda_i^R, \cdot) > 0$, and $f_i'(\lambda_i^R, \cdot) <0, \ \forall \lambda_i^R \in \overline{S_i}$, then the unique best response for player $i$ satisfies the implicit equation $b_i= \min\{-\frac{f_i(b_i, \cdot)}{f'_i(b_i,\cdot)}, \ \mu_i^R\}\in S_i$; and
    \item if $f'_i(\lambda_i^R,\cdot) \ge 0$, for every $\lambda_i^R \in S_i$, then $b_i=\mu_i^R$.
\end{enumerate}
\end{corollary}
\medskip
\begin{proof}
We only establish the first statement of the corollary. The other statements are established in the proof of Lemma~\ref{lemma:best_response_mapping}. We have already established in Lemma ~\ref{lemma:best_response_mapping} that if $f_i(\lambda_i^R,\cdot)\le 0$, then for every $\lambda_i^R \in S_i$, the expected utility $\tilde{u}_i$ is maximized for $b_i=0$. We now establish the ``only if" part. Recall from~\eqref{eq:7} that
\begin{equation*}
    \tilde{u}_i(\lambda_i^R, \lambda_{-i}^R)=\mu_i^Sr^S +\lambda_i^Rf_i(\lambda_i^R, \lambda_{-i}^R).
\end{equation*}
Let $b_i=0$ be the best response for player $i$ for a fixed $\lambda_{-i}^R$. If there exists $b \in S_i$, such that $f_i(b, \cdot) > 0$, then $\tilde{u}_i(b, \cdot) > \tilde{u}_i(b_i, \cdot)$, and $b_i=0$ cannot be a best response. Hence, $b_i=0$ is the best response for player $i$, if and only if, $f_i(\lambda_i^R,\cdot)\le 0$, for every $\lambda_i^R \in S_i$.

 We now show that if $f_i(\lambda_i^R,\cdot)\le 0$, for every $\lambda_i^R \in S_i$, then $f'_i(\lambda_i^R,\cdot)< 0$, for every $\lambda_i^R \in S_i$. Since $f_i$ is strictly concave in $x$ (from Lemma~\ref{lemma:incentive_function}) and $f_i(\mu_i^R, \ 0) = f_i(\mu_T^S-a_i \mu_i^R) > 0$ by Assumption (A3), there exist $\gamma_1, \gamma_2 \in \real$ such that  $\gamma_1 < \mu_T^S-a_i \mu_i^R < \gamma_2$ and $f_i(x)>0$ if and only if $x \in (\gamma_1, \gamma_2)$.  
 
 If $f_i(\lambda_i^R, \lambda_{-i}^R) = f_i(x)\le 0$ for each $\lambda_i^R \in S_i$ and for a given $\lambda_{-i}^R$, then for each $\lambda_i^R \in S_i$, either $x \le \gamma_1$, or $x \ge \gamma_2$.  Suppose $x \ge \gamma_2$, for each $\lambda_i^R \in S_i$. However, for $\lambda_i^R = \mu_i^R$, $x = \mu_T^S - a_i \mu_i - \sum_{j \ne i} a_j \lambda_j^R \le \mu_T^S - a_i \mu_i < \gamma_2$, which is a contradiction. Hence, $x \le \gamma_1$, for each $\lambda_i^R \in S_i$. 
 
Finally, from strict concavity of $f_i$, $f_i$ is increasing  in $x$ for $x \le \gamma_1$. Equivalently, $f_i$ is decreasing in $\lambda_i^R$, i.e., $f'_i(\lambda_i^R,\cdot)< 0$, for every $\lambda_i^R \in S_i$. 
\end{proof}

\begin{theorem}[\bit{Berge Maximum Theorem, adapted from~\cite{berge1997topological}}]\label{thm:berge}
Let $\tilde{u}_i : S_i \times S_{-i} \mapsto \mathbb {R}$ be a continuous function on $S_i \times S_{-i}$, and $C: S_{-i} \mapsto S_{i}$ be a compact valued correspondence such that $C(\lambda_{-i}^R) \neq \emptyset$ for all $\lambda_{-i}^R \in S_i$. Define $\tilde{u}_i^* : S_{-i} \mapsto \mathbb{R}$ by

\begin{equation*}
    \tilde{u}_i^*(\lambda_{-i}^R)=\max\{ \tilde{u}_i(\lambda_i^R, \lambda_{-i}^R) \  | \  \lambda_i^R \in C(\lambda_{-i}^R) \},
\end{equation*}

and $b_{i}: S_{-i} \mapsto S_i$ by

\begin{equation*}
    b_{i}(\lambda_{-i}^R)=\argmax\{ \tilde{u}_i(\lambda_i^R, \lambda_{-i}^R) \  | \  \lambda_i^R \in C(\lambda_{-i}^R) \}.
\end{equation*}

If $C$ is continuous at $\lambda_{-i}^R$, then $\tilde{u}_i^*$ is continuous and $b_{i}$ is upper hemicontinuous with nonempty and compact values. Furthermore, if $\tilde{u}_i$ is strictly quasiconcave in $\lambda_i^R \in S_i$ for each $\lambda_{-i}^R$  and $C$ is convex-valued, then $b_{i}(\lambda_{-i}^R)$ is single-valued, and thus is a continuous function.
\end{theorem}

\begin{lemma}[\bit{Continuity of best response mapping}]\label{lemma:best_response_continuous}
For the CPR game $\Gamma$, under Assumptions (A1-A3), the best response mapping $b_i(\lambda_{-i}^R)$ is continuous for each $\lambda_{-i}^R \in S_{-i}$.
\end{lemma}

\begin{proof}
Let $z(\lambda_{-i}^R): S_{-i} \mapsto [\frac{\mu_T^S - \sum_{j \in \mc{N}, j \neq i}a_j\mu_j^R}{a_i}, \  \frac{\mu_T^S}{a_i}]$ be defined by
\begin{equation}{\label{eq:15}}
    z(\lambda_{-i}^R) := \frac{\mu_T^S - \sum_{j \in \mc{N}, j \neq i}a_j\lambda_j^R}{a_i}.
\end{equation}
The mapping $z(\lambda_{-i}^R)$ represents an upper bound on the value of $\lambda_i^R$ above which the system constraint~\eqref{eq:3} is violated. Therefore, for each $\lambda_i^R \in [z(\lambda_{-i}^R), \infty)  \cap S_i$,  from~\eqref{eq:def-fi}, we get
\begin{align}\label{eq:16}
\!\!\!\! \!  f_i=-h_ir^S < 0, \text{and } f_i'=-r^Rp'\le 0, 
\end{align}

\noindent
{where the latter follows from monotonicity of $p$ (Assumption (A2)).}

The mapping $z(\lambda_{-i}^R)$ {defined in~\eqref{eq:15}} is continuous on $S_{-i}$ and linearly decreasing in $\lambda_j^R,$ {for every $ j \in \mc{N}\setminus\{i\}$.} Therefore, to establish the continuity of the best response mapping $b_i(\lambda_{-i}^R)$ on $S_{-i}$, it is sufficient to show that $b_i(\lambda_{-i}^R) = \phi(z(\lambda_{-i}^R))$, for some continuous function 
$\map{\phi}{ [\frac{\mu_T^S - \sum_{j \in \mc{N}, j \neq i}a_j\mu_j^R}{a_i}, \frac{\mu_T^S}{a_i}]}{[0, \  \mu_i^R]}$.
To this end, we show that for each fixed value of $z(\lambda_{-i}^R)$, $b_i$ is unique and varies continuously with $z(\lambda_{-i}^R)$. 

Let $\map{\hat{\lambda}^+}{S_{-i}}{[0, \ \mu_i^R]}$ be defined by
\begin{equation}\label{eq:sup}
\hat{\lambda}^+(\lambda_{-i}^R) = \begin{cases}
0, & \!\!\!\!\!\!\!\!\!\!\!\!\!\!\!\!\!\!\!\!\!\!\!\!\!\!\!\! \text{if } f_i(\lambda_i^R, \lambda_{-i}^R) \le 0, \forall \lambda_i^R \in S_i, \\
\sup\{\lambda_i^R \in S_i | \ f_i>0\}, & \text{otherwise}.
\end{cases}
\end{equation}
{The mapping $\hat{\lambda}^+(\lambda_{-i}^R)$, when non-zero, represents the maximal admissible review admission rate for player $i$, that yields her a positive incentive to review the tasks}. Fig.~\ref{fig:Picture3} shows the best response of player $i \in \mc{N}$ for the three possible cases of $\hat{\lambda}^+$.

\begin{figure}
\centering
\begin{subfigure}[b]{0.32\linewidth}
	    \centering
        \includegraphics[width=1\linewidth, height=1.1\linewidth, keepaspectratio]{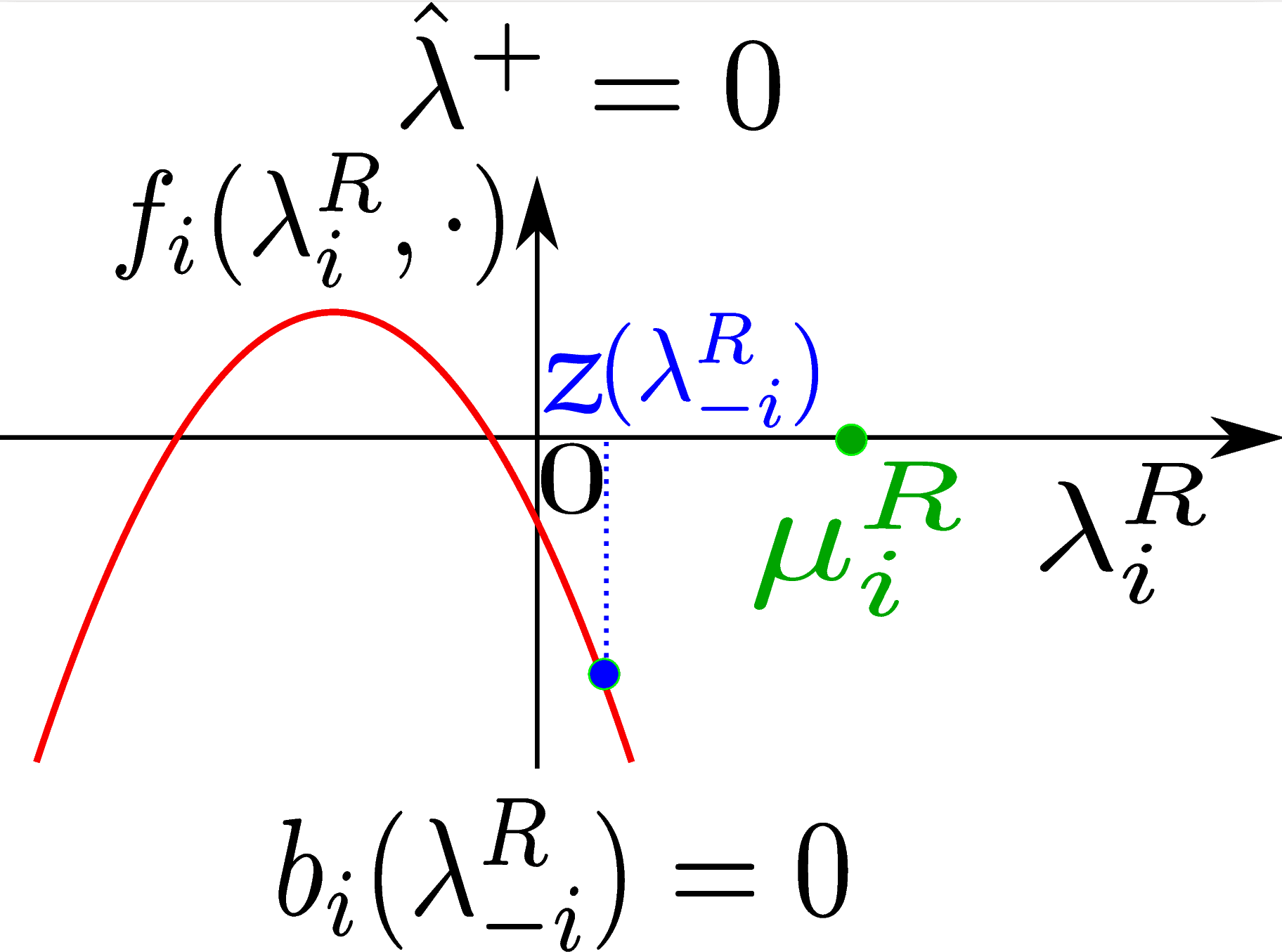}
        \caption{}
        \label{d}
    \end{subfigure}
\begin{subfigure}[b]{0.32\linewidth}
	    \centering
        \includegraphics[width=1\linewidth, height=1.1\linewidth, keepaspectratio]{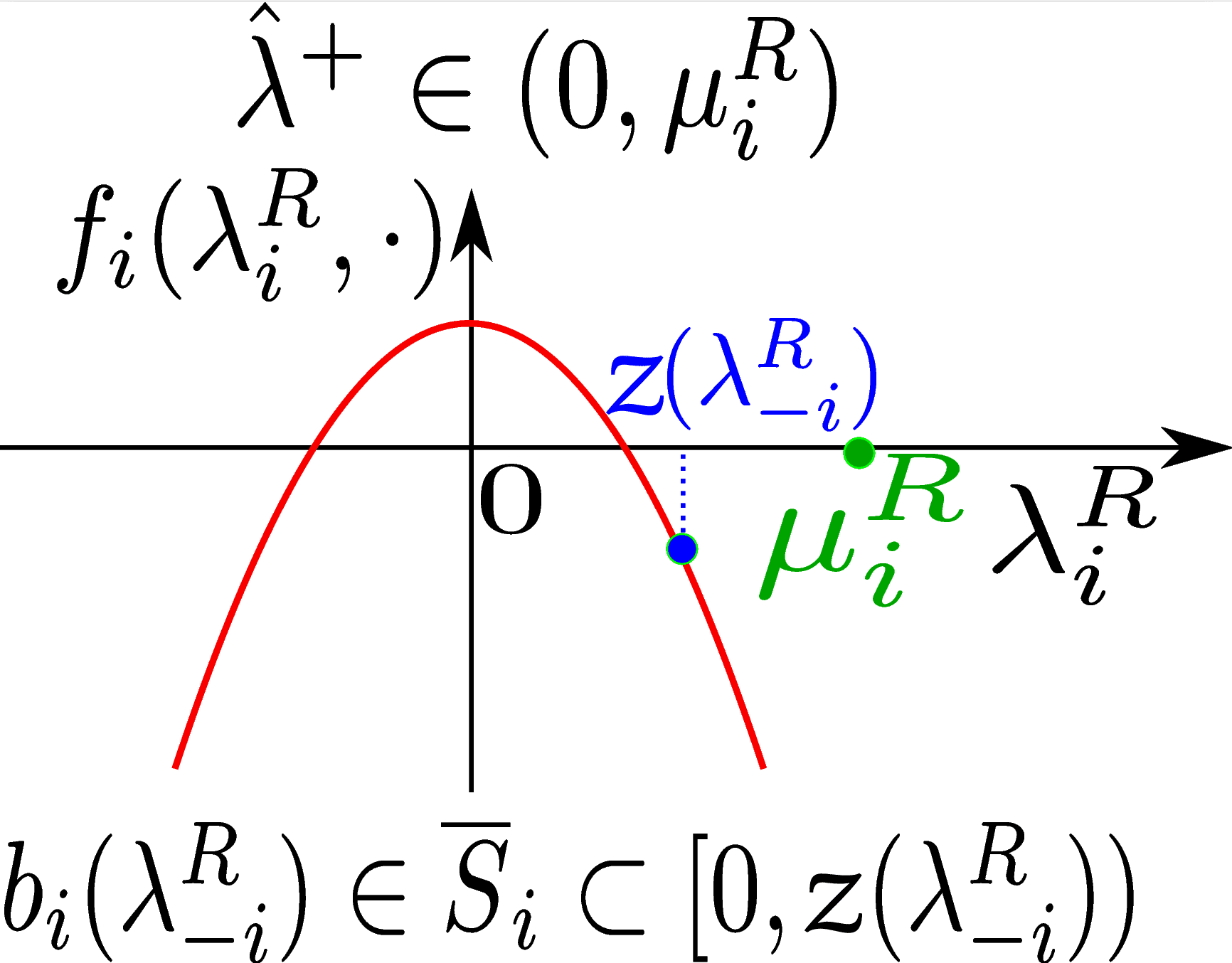}
        \caption{}
        \label{e}
\end{subfigure}
\begin{subfigure}[b]{0.32\linewidth}
	    \centering
        \includegraphics[width=1\linewidth, height=1.1\linewidth, keepaspectratio]{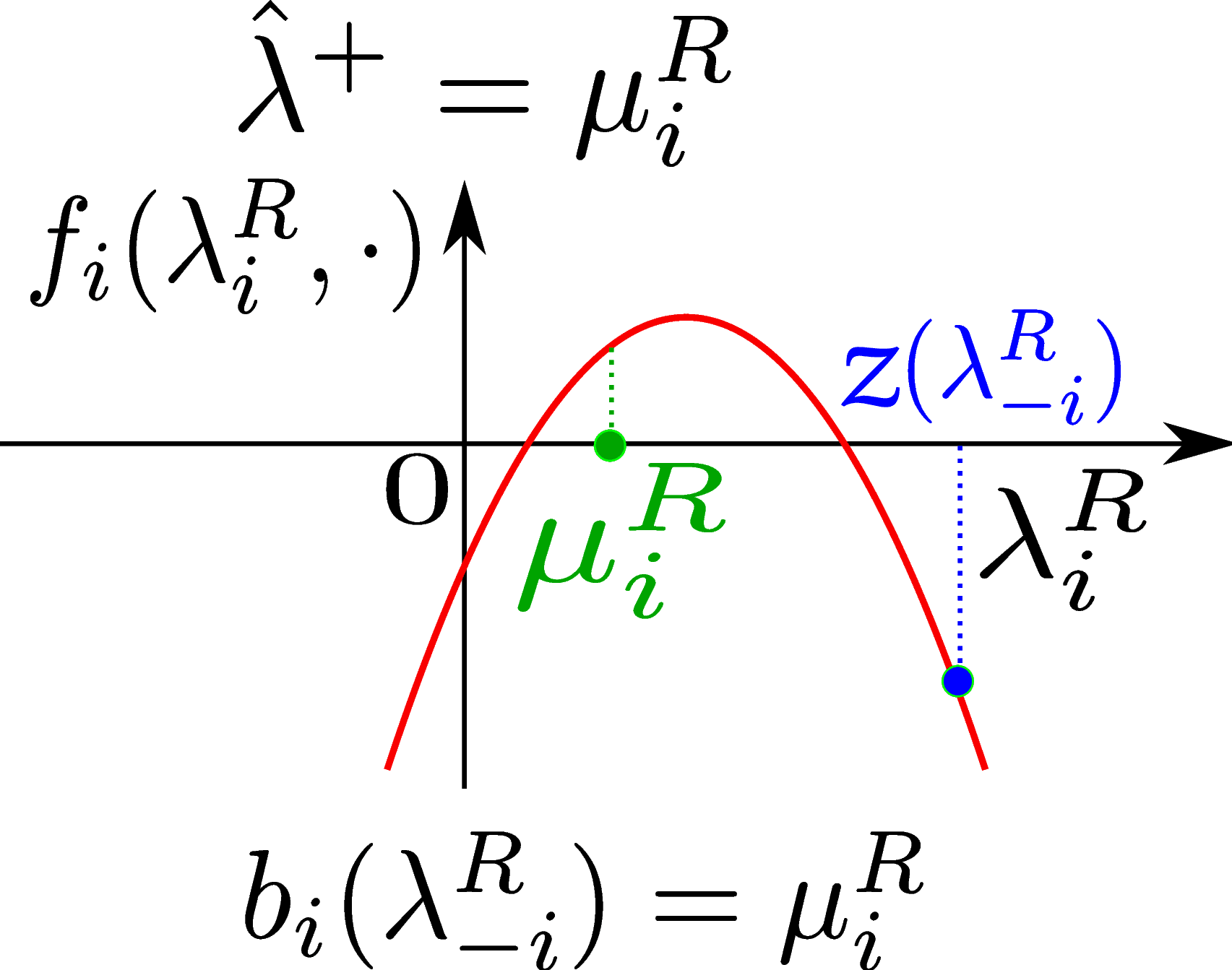}
        \caption{}
        \label{f}
\end{subfigure}
    \caption{\footnotesize {Best response of player $i$ with varying $\hat{\lambda}^+$. The red curve shows different possibilities for strictly concave incentive function $f_i(\lambda_i^R)$ w.r.t $\lambda_i^R \in S_i = [0, \mu_i^R]$, based on the value of $\lambda_{-i}^R$. In (a), $f_i <0$ and $f'_i<0$ for all $\lambda_i^R \in S_i$; in (b), there exists a subset of $S_i$ where $f_i > 0$ and $f'_i<0$; and in c) $f'_i\ge0$ for any $\lambda_{-i}^R$. At $z(\lambda_i^R)$ (represented by blue), $f_i <0$ and $f'_i<0$.
    a) For $\hat{\lambda}^+=0, \; b_i(\lambda_{-i}^R)=0 $; b) for $\hat{\lambda}^+ \in (0, \ \mu_i^R) \;, \; b_i(\lambda_{-i}^R) \in \overline{S_i}$;
    and c) for $\hat{\lambda}^+=\mu_i^R\; , \; b_i(\lambda_{-i}^R)=\mu_i^R$. }}
    \label{fig:Picture3}
\end{figure}

\noindent
{\bf Case 1:} {$z(\lambda_{-i}^R) \le 0$}. 
{From~\eqref{eq:16} and statement (i) of Corollary~\ref{corollary1}, $b_i(\lambda_{-i}^R)=0$ is the unique (continuous) best response for player $i$. }

\noindent
{\bf Case 2:} {$z(\lambda_{-i}^R) >0$. }
From~\eqref{eq:16}, $f_i <0$ and $f_i'<0$ for any $\lambda_i^R \ge z(\lambda_{-i}^R)$. Hence, $\hat{\lambda}^+ < z(\lambda_{-i}^R)$. {Now we consider three cases based on the value of $\hat{\lambda}^+$.}

\noindent
{
{\it Case 2.1:} $\hat \lambda^+=0$. In this case, $ f_i\leq0,$ for every $\lambda_i^R \in S_i $, and therefore, from statement (i) of Corollary~\ref{corollary1}, $b_i=0$ is the unique best response and continuity holds trivially. 
}

\noindent
{
{\it Case 2.2:} $\hat{\lambda}^+ \in (0, \  \mu_i^R)$. In this case, for any $\lambda_{-i}^R \in S_{-i}$, there exists an interval $\overline{S_i} \subset [0, \ \hat{\lambda}^+]$ such that $f_i>0$ and $f_i'<0,$ for every $\lambda_i^R \in \overline{S_i}$.} Here, $f_i'<0$ follows from the fact that the supremum in~\eqref{eq:sup} corresponds to the decreasing segment of $f_i$. From statement (ii) of Corollary~\ref{corollary1}, there exists a unique $b_i(\lambda_{-i}^R) \in \overline{S_i}$ that maximizes $\tilde u_i$. Application of Berge maximum theorem~\cite{berge1997topological}, yields the continuity of the unique maximizer. 

Since, $b_i < \hat \lambda^+$, it follows that if $\hat \lambda^+ \to 0^+$, then $b_i \to 0^+$. Hence, the continuity holds at $\hat \lambda^+ =0$. 


\noindent
{
{\it Case 2.3:} $\hat{\lambda}^+= \mu_i^R$. Since $\hat{\lambda}^+ < z(\lambda_{-i}^R)$, $z(\lambda_{-i}^R) \in (\mu_i^R, \frac{\mu_T^S}{a_i}]$.} If $f'(\mu_i^R, \lambda_{-i}^R)<0,$ then 
the continuity follows analogously to Case 2.2. 
Now consider the case $f'(\mu_i^R, \lambda_{-i}^R) = -\delta$, for $\delta >0$. 
Since $f_i$ is concave in $\lambda_i^R$ and its derivative is decreasing, there exists $\epsilon >0$ such that $f' <0$ for $\lambda_i \in (\mu_i^R-\epsilon, \mu_i^R]$. Since $f_i(\lambda_i^R, \cdot)$ is strictly concave in $\lambda_i^R$ (Lemma~\ref{lemma:incentive_function}), there exists at most one point $\lambda_i^R$, such that $f'(\lambda_i^R, \cdot)=0$. Therefore, in the limit $\delta \to 0^+$, $\epsilon \to 0^+$. Hence, in this 
limiting case $\overline{S_i} = (\mu_i^R-\epsilon, \ \mu_i^R]$, where $\epsilon \to 0^+$, and the best response $b_i(\lambda_{-i}^R) \in \overline{S_i} = (\mu_i^R-\epsilon, \ \mu_i^R]$ converges to $\mu_i^R$.

If  $f'(\mu_i^R, \lambda_{-i}^R) \geq 0$, then it follows from strict concavity of $f_i$ that $f'(\lambda_i^R, \lambda_{-i}^R)\geq 0,$  for every $ \lambda_i^R \in S_i$.
 Using statement (iii) of Corollary~\ref{corollary1}, $b_i(\lambda_{-i}^R) = \mu_i^R$ is the unique (continuous) best response.
 
 Note that when $z(\lambda_{-i}^R)= \frac{\mu_T^S}{a_i}$, i.e., when no other player reviews any task ($\lambda_{-i}^R=0$), from Assumption (A3), $f_i(\mu_i^R,0)>0$ and therefore $b_i(\lambda_{-i}^R)= \mu_i^R$.
 Hence, $b_i(\lambda_{-i}^R)$ is continuous for every $ z(\lambda_{-i}^R)$, and therefore, is continuous for $\lambda_{-i}^R \in S_{-i}$.
\end{proof}

\medskip
\textit{Proof of Theorem~\ref{thm:thm1}:}
To prove the existence of a PNE, define a mapping $M: S \mapsto S$ as follows:
\begin{equation}{\label{eq:17}}
    M(\lambda_1^R,\ \lambda_2^R,..., \ \lambda_N^R)=(b_1(\lambda_{-1}^R), \ b_2(\lambda_{-2}^R),...,\ b_N(\lambda_{-N}^R)). \\
\end{equation}

The mapping $M$ is unique (Lemma~\ref{lemma:best_response_mapping}) and continuous (Lemma~\ref{lemma:best_response_continuous}), and maps the compact convex set $S \ (S_i$ is convex and compact, $\forall i \in \mc{N})$ to itself. Hence, application of Brouwer's fixed point theorem~\cite[Appendix C]{basar1999dynamic} yields that there exists a strategy profile $\lambda^R=\{{\lambda_i^R}^*\}_{i \in \mc{N}} \in S$ which is invariant under the best response mapping and therefore is a PNE of the game. \oprocend 

\subsection{Proof of Corollary~\ref{corollary2} [PNE]:}\label{Corollary_appendix}
Since PNE is a best response which remains invariant under the best-response mapping $M$ given by~\eqref{eq:17}, Corollary~\ref{corollary2} is a direct consequence of Corollary~\ref{corollary1} with a simplification that $f_i'(\lambda_i^{R^*},\lambda_{-i}^{R^*}) < 0$ at PNE.
Therefore, to prove Corollary~\ref{corollary2}, it is sufficient to show statement (i), i.e.  $f_i'(\lambda_i^{R^*},\lambda_{-i}^{R^*}) < 0$ for any player $i \in \mc{N}$ at PNE, which we prove by contradiction. Let there exist a player $j$ such that $f_j'(\lambda_j^{R^*},\lambda_{-j}^{R^*}) \ge 0$ at PNE. 
From~\eqref{eq:first derivative}, it can be seen that the sign of $f_i'$ remains the same for all players at a PNE. Therefore, $f_j' \geq 0$ implies $f_i'\geq 0$ for all $i \in \mc{N}$ at that PNE. In such a case,~\eqref{eq:13} implies that the expected utility of each player with $\lambda_i^{R^*}>0$ (therefore, $f_i>0$) is increasing in $\lambda_i^R$ at that PNE, and therefore, each of these players can improve their expected utility by unilaterally increasing their review admission rate. Therefore $\lambda^{R^*}$ cannot be a PNE, which is a contradiction. Hence, $f_i'(\lambda_i^{R^*},\lambda_{-i}^{R^*}) < 0$ for any player $i \in \mc{N}$ at a PNE, and the corollary follows. \oprocend

\subsection{Proof of Proposition~\ref{proposition:Nash_Equilibrium} [Structure of PNE]} \label{structure_appendix}
 Let $\lambda_{k_1}^{R^*}$ and $\lambda_{k_2}^{R^*}$ be the review admission rates at a PNE for players $k_1$ and $k_2$, respectively,  with $h_{k_1} \le h_{k_2}$. By proving $a_{k_1}\lambda_{k_1}^{R^*} \ge a_{k_2}\lambda_{k_2}^{R^*}$, $\lambda_{k_1}^{R^*} \ge \lambda_{k_2}^{R^*}$ is established trivially since $a_{k_1} \le a_{k_2}$.
 We assume $a_{k_1}\lambda_{k_1}^{R^*} < {a_{k_2}}\lambda_{k_2}^{R^*}$ and prove the first statement by establishing a contradiction argument using two cases discussed below. Furthermore, the proof of the second statement is contained within Case 1 below.
 
 \medskip
\noindent
{\bf Case 1:} $\lambda_{k_1}^{R^*}=0$. 

From  statement (ii)  of Corollary~\ref{corollary2}, $f_{k_1} (\lambda_{k_1}^{R^*}, \lambda_{-k_1}^{R^*}) \le 0$. From~\eqref{eq:def-fi}, the incentives $f_{k_1}$ and $f_{k_2}$ for players $k_1$ and $k_2$ at a PNE satisfies:
\begin{equation*}
f_{k_2} = f_{k_1} + (h_{k_1}-h_{k_2})r^S\le0.
\end{equation*}
Therefore, utilizing statement (ii)  of Corollary~\ref{corollary2} again implies $\lambda_{k_2}^{R^*}=0$, which is a contradiction. This case also proves the second statement.

\medskip
\noindent
{\bf Case 2:} $\lambda_{k_1}^{R^*}>0$.

{By assumption, $a_{k_1}\lambda_{k_1}^{R^*} < a_{k_2}\lambda_{k_2}^{R^*}$,  from  statement (iii)  of Corollary~\ref{corollary2}, $\lambda_{i}^{R^*}$, where $i \in \{ k_1, k_2\}$, satisfy the implicit equation

\begin{align*}
    \lambda_{i}^{R^*}&=\min \Bigg\{-\frac{f_{i}(\lambda_{i}^{R^*}, \  \lambda_{-i}^{R^*})}{f'_{i}(\lambda_{i}^{R^*}, \ \lambda_{-i}^{R^*})}, \ \mu_{i}^R \Bigg \}.
\end{align*}

\medskip
We assume that $\lambda_{k_1}^{R^*} < \mu_{k_1}^R$, and therefore,  $\lambda_{k_1}^{R^*}=-\frac{f_{k_1}}{f'_{k_1}}$. Using~\eqref{eq:def-fi} and~\eqref{eq:first derivative}, we get 

\begin{align*}
    a_{k_2}\lambda_{k_2}^{R^*} &= \min \Bigg\{-a_{k_2}\frac{f_{k_2}}{f'_{k_2}}, \ a_{k_2}\mu_{k_2}^R \Bigg \} \\
    &\le -a_{k_2}\frac{f_{k_2}}{f'_{k_2}}= -a_{k_1}\frac{f_{k_1} + (h_{k_1}-h_{k_2})r^S}{f'_{k_1}}  \le a_{k_1}\lambda_{k_1}^{R^*},  
\end{align*}
which is a contradiction. Hence, if $\lambda_{k_1}^{R^*} < \mu_{k_1}^R$, then $a_{k_1}\lambda_{k_1}^{R^*} \ge a_{k_2}\lambda_{k_2}^{R^*}$ and $\lambda_{k_1}^{R^*} \ge \lambda_{k_2}^{R^*}$ for each $k_2 > k_1$.} \oprocend

\subsection{Proof of Theorem 2 [Uniqueness of PNE]}\label{Uniqueness_appendix}

Suppose that the CPR game $\Gamma$ has multiple PNEs. We define the support of a PNE as the total number of players with non-zero review admission rate. 
Let $\mr{PNE}_1 = {\lambda^1}= [{\lambda_1^1},\ {\lambda_2^1}, \ldots, \  {\lambda_N^1}]$ and $\mr{PNE}_2 = {\lambda^2}= [{\lambda_1^2},\ {\lambda_2^2}, \ldots, \ {\lambda_N^2}]$, be two different PNEs with distinct supports $m_1$ and $m_2$, respectively. For brevity of notation, we have removed the superscript $R$ from the two PNEs and replaced it by their unique identifier. Without loss of generality, let $m_2>m_1$. Let $x^1=\mu_T^S-\sum_{i=1}^{N}{a_i}\lambda_i^1$ and $x^2=\mu_T^S-\sum_{i=1}^{N}{a_i}\lambda_i^2$ be the slackness parameters at $\mr{PNE}_1$ and $\mr{PNE}_2$, respectively.

We prove the uniqueness of PNE using a six step process.

\medskip

\noindent
 \textbf{Step 1:}  \textit{We first show that if there exists two different PNEs with distinct supports $m_1$ and $m_2$ ($m_1 < m_2$), then $x^1<x^2$.}

If $m_1$ and $m_2$ are the supports of $\mr{PNE}_1$ and $\mr{PNE}_2$, respectively, then $\lambda_i^1=0$ and $\lambda_j^2=0$, for each $i>m_1$, and $j>m_2$, respectively (Proposition~\ref{proposition:Nash_Equilibrium}). Additionally, $\lambda_i^1>0$, and $\lambda_j^2>0$, for each $i \leq m_1$, and $j \leq m_2$. Hence, $m_2>m_1$ implies $\lambda_{m_2}^1=0$, while $\lambda_{m_2}^2>0$. 

From statement (i) of Corollary~\ref{corollary1}, $b_i=0$, if and only if $f_i\le 0$ and ${f_i}'< 0$ (equivalently $\frac{df_i}{dx}>0$) for all $\lambda_i^R \in S_i$. Therefore, $\lambda_{m_2}^1=0$ implies $f_{m_2}^1:= f_{m_2}(\lambda^1)\le0$ and $\frac{df_{m_2}}{dx}>0$ everywhere, while $\lambda_{m_2}^2>0$ implies $f_{m_2}^2:= f_{m_2}(\lambda^2)>0$. Since $f_{m_2}^2>0>f_{m_2}^1$ and $\frac{df_{m_2}}{dx}>0$ everywhere, it follows that $x^1<x^2$.
 
 \medskip
 
 \noindent
 {\textbf{Step 2:} \textit{We now show that $x^1>x^2$ using Steps 2-5, which is a contradiction to the result of Step 1}, and consequently $m_1=m_2$.}
 
From statement (iii) of Corollary~\ref{corollary2}, the review admission rate of any player $i$, $i\le m_1$, at $\mr{PNE}_k$, $k \in \{1,2\}$, satisfies

\begin{equation}{\label{eq:lambda}}
    \lambda_i^k=\min{\Bigg\{-\frac{f_i^k}{{f_i^k}'}, \ \mu_i^R \Bigg\}}.
\end{equation}

\medskip

\noindent 
  \textbf{Step 3:} \textit{We show that $f_{i}^2>f_{i}^1$ for any player $i$, $i\le m_1$.}
  
From~\eqref{eq:def-fi}, the incentives $f_i$ and $f_j$ for any two distinct players $i$ and $j$ with $j>i$ at a $\mr{PNE}_k$, $k \in \{1,2\}$ satisfies:
\begin{equation*}{\label{eq:numerator}}
    f_i^k-f_j^k= (h_j-h_i)r^S >0, \ \forall j>i.
\end{equation*}
Notice that the right hand side of above equation is independent of $\lambda_i^R$ and therefore, a constant for both PNEs.
Hence, for every $i< m_2$
\begin{equation*}
    f^1_{i}-f^1_{m_2} = f^2_{i}-f^2_{m_2}.
\end{equation*}
Therefore, $f_{m_2}^2>f_{m_2}^1$ implies $f_{i}^2>f_{i}^1$, for every $i \le m_1 < m_2$.

 \noindent 
 \textbf{Step 4:} \textit{We show that ${{f'}_i^1}<{{f'}_i^2}$, for every player $i$, $i\le m_1$.}
 
Recall that $f_i$ is strictly concave in $x$ (Lemma~\ref{lemma:incentive_function}). Therefore, $x^1<x^2$ (Step 1) implies $\frac{df_i^1}{dx}>\frac{df_i^2}{dx}$. 
Therefore, from~\eqref{eq:first derivative}, ${{f'}_i^1}<{{f'}_i^2}$,  for any player $i$, $i\le m_1$.

\medskip {

\noindent 
\textbf{Step 5:} \textit{We now show that $x^1 > x^2$, which is a contradiction to result of Step 1, and consequently $m_1=m_2$.}}

Since for all players $i$, $i \le m_1$, $f_{i}^2>f_{i}^1$ (Step 3) and $-{{f'}_i^1}>-{{f'}_i^2}$ (Step 4),~\eqref{eq:lambda} implies $\lambda_i^2 \ge \lambda_i^1$, for each $i\le m_1$. Therefore, $\sum_{i=1}^{N}{a_i}\lambda_i^2> \sum_{i=1}^{m_1}{a_i}\lambda_i^2 \ \ge \ \sum_{i=1}^{m_1}{a_i}\lambda_i^1=\sum_{i=1}^{N}{a_i}\lambda_i^1$, which implies $x^1>x^2$, {which is a contradiction to result of Step 1. Hence, $m_1 = m_2$
}
{

\noindent 
\textbf{Step 6:} \textit{We now show the value of slackness parameter $x$ at any PNE is unique}.

Steps 1 to 5 show that, at a PNE, the number of players with non-zero review admission rate are unique. Therefore, let $m$ be the identical support for $\mr{PNE}_1$ and $\mr{PNE}_2$. Without loss of generality, let $x_1>x_2$. 

Let $g_i : \mathbb{R} \mapsto  \mathbb{R}$, for $i \le m$, be  defined by

\begin{equation*}
    g_i(x)=-\frac{f(x)}{f'(x)}.
\end{equation*}

Differentiating $g_i(x)$ w.r.t $x$, we get

\begin{equation*}
    \frac{d g_i(x)}{d x}=\frac{{(\frac{d f_i(x)}{d x})}^2 - f_i(x)\frac{d^2f_i(x)}{d x^2}}{a_i {(\frac{d f_i(x)}{d x})}^2}.
\end{equation*}

Recall from statement (iii) of Corollary~\ref{corollary2} that players have non-zero review admission rate at PNE, if and only if, $f_i>0$ at PNE. Strict concavity of $f_i$ (Lemma~\ref{lemma:incentive_function}) implies $\frac{d g_i(x)}{d x}>0$. Consequently, at PNE, the review admission rate for any player $i$, $i\le m$, is increasing with $x$. Therefore, assumption $x^1 >x^2$ implies $\lambda_i^1 \ge \lambda_i^2$, for each player $i \le m$. Consequently,  $x^1 = \mu_T^S-\sum_{i=1}^{m}{a_i}\lambda_i^1 \le \mu_T^S-\sum_{i=1}^{m}{a_i}\lambda_i^2 = x^2$, which is a contradiction. Therefore, $x^1=x^2$.

\medskip

\noindent 
We now show the uniqueness of PNE. Steps 1 to 6 show that, at a PNE, the number of players with non-zero review admission rate and the slackness parameter $x$ are unique. Therefore, the first order conditions~\eqref{eq:lambda} give the unique review admission rate for each player $i$ for unique slack parameter $x$, thereby implying uniqueness of PNE.}
\oprocend 

\subsection{Proof of Lemma~\ref{lemma:best_response_non_increasing} [Non-increasing best response]}\label{non_increasing_appendix}

We prove this lemma by considering the three cases of the best response mapping in Lemma~\ref{lemma:best_response_mapping} (Appendix~\ref{Existence_appendix}):

\medskip
 \noindent \textbf{Case 1:} {$b_i=0$. Recall that $x =\mu_T^S-\sum_{i=1}^{N}{a_i}\lambda_i^R$. In this case, from statement (i) of Corollary~\ref{corollary1}, $f_i\le 0$ and  $f_i'<0$  (equivalently, $\frac{d f_i}{d x} >0)$,  for all $\lambda_i^R \in S_i$. Since $x$ can be re-written as $x=\mu_T^S-a_i\lambda_i^R -\sigma_i(\lambda_{-i}^R)$, therefore $\frac{d f_i}{d x} >0$ implies $\frac{\partial f_i}{\partial \sigma_i}<0$. Hence, with increase in $\sigma_i(\lambda_{-i}^R)$, $b_i=0$ remains the best response.}
 
 \medskip
  \noindent \textbf{Case 2:} $b_i=-\frac{f_i(b_i,\sigma_i(\lambda_{-i}^R))}{f_i'(b_i,\sigma_i(\lambda_{-i}^R))}$. In this case, from statement (ii) of Corollary~\ref{corollary1}, $b_i \in \overline{S_i}$ such that $f_i>0$ and  $f_i'<0$, for every $\lambda_i^R \in \overline{S_i}$. Thus, 
  \begin{equation}
      \frac{d b_i}{d \sigma_i}=\frac{-{f_i'}^2+f_i''f_i}{a_i{f_i'}^2} <0. 
  \end{equation}
  Hence, $b_i$ is strictly decreasing in $\sigma_i(\lambda_{-i}^R)$.
  
  \medskip
  \noindent \textbf{Case 3:}
  $b_i=\mu_i^R$. Since $b_i \in S_i=[0, \ \mu_i^R]$, $b_i$ either decreases or remains constant with increase in $\sigma_i(\lambda_{-i}^R)$.   \oprocend 

\subsection{Proof of Theorem~\ref{thm:thm3} [Analytic bounds on PNE Inefficiency]}\label{bound_appendix}

We first establish the analytic upper bound on PoA, followed by upper bounds on $\eta_{TRI}$ and $\eta_{LI}$.  

Let $\mathcal{G}$ be the family of CPR games parameterized by the ratios of the maximum service and review admission rates of each player $i \in \mathcal{N}$. Therefore, the CPR game $\Gamma \in \mathcal{G}$, with the corresponding ratio for player $i$  given by $h_i$. Define a set of homogeneous CPR games $\mathcal{G}^H \subset \mathcal{G}$, in which each player has a constant ratio $\frac{\{\mu_i^S\}^H}{\{\mu_i^R\}^H} =: h$, and $\min_i\{\{\mu_i^R\}^H\} \ge \frac{\{\mu_T^S\}^H}{N(1+h)}$. The superscript $H$ is used to distinguish maximum service and review admission rates of the homogeneous game $\Gamma^H$ from the CPR game $\Gamma$.

For any CPR game in $\mathcal{G}$, PoA is given by:

\begin{equation}{\label{poa_eq}}
     PoA=\frac{(\Psi)_{SW}}{(\Psi)_{PNE}} = \frac{[\sum_{i=1}^{N}\mu_i^Sr^S+ \sum_{i=1}^{N}\lambda_i^Rf_i(x)]_{SW}}{[\sum_{i=1}^{N}\mu_i^Sr^S+ \sum_{i=1}^{N}\lambda_i^Rf_i(x)]_{PNE}}.
\end{equation}

We now provide an analytic upper bound on the PoA for the CPR game $\Gamma$ using following Lemmas.

\begin{lemma}[\bit{PNE solution for homogeneous CPR game}]\label{lemma:homogeneous PoA_solution}
For any homogeneous CPR game $\Gamma^H \in \mathcal{G}^{H}$, such that for each player $i \in \mathcal{N}$, $\frac{\{\mu_i^S\}^H}{\{\mu_i^R\}^H} = h$, and $\min_i\{\{\mu_i^R\}^H\} \ge \frac{\{\mu_T^S\}^H}{N(1+h)}$, each player participates in the review process with equal review admission rate $\lambda_i^H=\lambda_H$ at PNE. Let $\lambda_T^H$ be the total review admission rate at PNE for $\Gamma^H$. The unique PNE solution is given by $\lambda_H = \frac{\lambda_T^H}{N}$, where $\lambda_T^H = \frac{f(x)}{(1+h)\frac{d f}{d x}}$ and $x=\{\mu_T^S\}^H- (1+h)\lambda_T^H$. 
\end{lemma}

\begin{proof}
For the homogeneous CPR game $\Gamma^H$, each player has equal incentive $f(x)$ to review the tasks. If $f(x) \le 0$ at PNE, all players have $\lambda_i^H=0$ (statement (ii) of Corollary \ref{corollary2}) which contradicts assumption (A3). Hence, at PNE, each players has $\lambda_i^H >0$. 

Let $\min_i\{\{\mu_i^R\}^H\} \ge \frac{\{\mu_T^S\}^H}{N(1+h)}$ for $\Gamma^H$. At PNE, $x>0$. Let $\mathcal{D} \subseteq \mathcal{N}$ be a non-empty set of player indices such that for any $i \in \mathcal{D}$, $\{\mu_i^R\}^H \le -\frac{f(x)}{f'(x)}$. At PNE,
\begin{align*}
    \lambda_T^H &= \sum_{i \in \mathcal{D}}\{\mu_i^R\}^H  + \sum_{i \in \mathcal{N} \setminus \mathcal{D}} \frac{-f(x)}{f'(x)} \\
    &\ge N\min_i\{\{\mu_i^R\}^H \} \\
    &\ge \frac{\{\mu_T^S\}^H}{(1+h)}.
\end{align*}
Therefore, at PNE, 
\begin{align*}
    x&=\{\mu_T^S\}^H- (1+h)\lambda_T^H \\
    &\le \{\mu_T^S\}^H- (1+h)N \min_i\{\{\mu_i^R\}^H \}
    \le 0,
\end{align*}
which is a contradiction. Hence, $\mathcal{D}$ is an empty set and each player has equal review admission rate at PNE, given by $\lambda_i^H= \lambda_H = \frac{\lambda_T^H}{N}$. Hence, each player being a maximizer of their expected utility maximizes:

\begin{equation}
    \tilde{u}_i= \{\mu_i^S\}^Hr^S + \frac{\lambda_T^H}{N}f(x),
\end{equation}
 where $x=\{\mu_T^S\}^H- (1+h)\lambda_T^H$. Setting $\frac{\partial \tilde{u}_i}{\partial \lambda_T^H}=0$, we get $\lambda_T^H = \frac{f(x)}{(1+h)\frac{d f}{d x}}$.
\end{proof}

\begin{lemma}[\bit{PoA=1 for homogeneous CPR game}]\label{lemma:homogeneous PoA}
For any homogeneous CPR game $\Gamma^H \in \mathcal{G}^{H}$, such that for each player $i \in \mathcal{N}$, $\frac{\{\mu_i^S\}^H}{\{\mu_i^R\}^H} = h$, and $\min_i\{\{\mu_i^R\}^H\} \ge \frac{\{\mu_T^S\}^H}{N(1+h)}$, PoA=1. 
\end{lemma}

\begin{proof}
For homogeneous CPR game $\Gamma^H$,
social welfare function $\Psi^H$ in \eqref{eq:8} only depends on $\lambda_T^R$ and is given by:

\begin{equation}
    \Psi^H = \{\mu_T^S\}^Hr^S + \lambda_T^R f(x), 
\end{equation}
where $x=\{\mu_T^S\}^H- (1+h)\lambda_T^R$, and $f(x)$ is the uniform incentive function for each player. Note that  $\frac{d \Psi^H}{d \lambda_T^R} > 0$ when $\frac{d f}{d x} \le 0$, and $\frac{d^2 \Psi^H}{d {\lambda_T^R}^2}>0$ in the interval where $\frac{d f}{d x} > 0$.  It is easy to show that $\Psi^H$ is maximized by any ${\lambda_T^R}$ satisfying $\lambda_T^R =\frac{f(x)}{(1+h)\frac{d f}{d x}}$ obtained by setting $\frac{d \Psi^H}{d \lambda_T^R}=0$, and $\frac{d f}{d x} > 0$ at the maximizer. 

Let $\lambda_T^H$ be the total review admission rate at PNE for $\Gamma^H$. 
The unique PNE satisfies   $\lambda_T^H= \frac{f(x)}{(1+h)\frac{d f}{d x}}$ (Lemma \ref{lemma:homogeneous PoA_solution}), and hence, maximizes social welfare utility resulting in PoA$=1$.
\end{proof}

Corresponding to the CPR game $\Gamma$, construct a homogeneous game $\Gamma^H_N \in \mathcal{G}^H$ with
$\{\mu_i^S\}^H=\mu_i^S$ and $\{\mu_i^R\}^H= \frac{\mu_i^S}{h_N}$, for each player $i\in \mathcal{N}$. Note that for homogeneous players in $\Gamma^H_N$, $\frac{\{\mu_i^S\}^H}{\{\mu_i^R\}^H}=h=h_N$ , and  $\sum_{i=1}^{N}\{\mu_i^S\}^H= \mu_T^S$. Furthermore, the assumption $\min_i\{\mu_i^S\} > \frac{\mu_T^Sh_N}{N(1+h_N)}$ implies $\min_i\{\{\mu_i^R\}^H\}> \frac{\{\mu_T^S\}^H}{N(1+h)}$. Hence, $PoA=1$ for $\Gamma^H_N$ (Lemma~\ref{lemma:homogeneous PoA}).

To obtain analytic bounds on PoA, we now compute a lower bound on the social utility obtained at the unique PNE $\Psi^{\Gamma}_{PNE}$ for the CPR game $\Gamma$. In Lemma~\ref{lemma:homogeneous vs heterogeneous}, we show that the 
social utility obtained at the PNE $\Psi^H_{PNE}$ for the homogeneous game $\Gamma_N^H$ lower bounds $\Psi^{\Gamma}_{PNE}$. For any $x \in [0, \ \mu_T^S]$, homogeneous players with ratio $h_N$ in  $\Gamma_N^H$ have a lower cumulative incentive ($\sum f$) to review tasks than players in $\Gamma$, and therefore, have a lower social utility at PNE, i.e., $\Psi^{\Gamma}_{PNE} \ge \Psi^H_{PNE}$. We further lower bound $\Psi^{\Gamma}_{PNE}$ by computing a lower bound on $\Psi^H_{PNE}$.

\begin{lemma}[\bit{Lower bound for social welfare at PNE}]\label{lemma:homogeneous vs heterogeneous}
Let $\Gamma^H_N$ be a homogeneous game corresponding to CPR game $\Gamma$ with each player $i\in \mathcal{N}$ having $\{\mu_i^S\}^H=\mu_i^S$ and $\{\mu_i^R\}^H=\frac{\mu_i^S}{h_N}$. Let $\Psi_{PNE}^{\Gamma}$ and $\Psi_{PNE}^H$ be the social welfare functions for $\Gamma$ and $\Gamma^H_N$, respectively, evaluated at their unique PNEs. Then  $\Psi^{\Gamma}_{PNE} \ge \Psi^H_{PNE} \ge \mu_T^Sr^S +\frac{\mu_T^S- \overline{x}}{a_N}f_N(\overline{x})$, where $\overline{x}$ is the unique maximizer of $f_i$, i.e. $\frac{d f_i}{d x}(\overline{x})=0$.
\end{lemma}

\begin{proof}
Let $\lambda^* = [\lambda_1^*, \ldots, \lambda_N^*]$ and $\lambda^H = [\lambda_H, \ldots, \lambda_H]$ be the unique PNEs for the CPR games $\Gamma$ and $\Gamma^H_N$, respectively. Let $x^*=\mu_T^S - \sum_{i=1}^{N} a_i\lambda_i^*$ and $x^H=\mu_T^S - Na_N\lambda_H$ be their slackness parameters at respective PNEs.

\noindent 
\textbf{Step 1:} \textit{We show that $x^H \ge x^*$ using contradiction.} 

Let $x^* > x^H$. Recall that at PNE, $\frac{d f_i}{d x} > 0$ (Corollary~\ref{corollary2}).  Using strict concavity of $f_i$ (Lemma~\ref{lemma:incentive_function}), we have $\frac{d f_i}{d x}(x^H) > \frac{d f_i}{d x}(x^*) >0$. Therefore, $\frac{f_N(x^*)}{\frac{d f_N}{d x}(x^*)} > \frac{f_N(x^H)}{\frac{d f_N}{d x}(x^H)}= a_N\lambda_N^H$.
Recall that $f_1(x) \ge \cdots \ge f_N(x)$ for any $x$, and $\frac{d f_i}{d x}$ is independent of $i$. Hence, $\frac{f_i(x^*)}{\frac{d f_i}{d x}(x^*)} > a_N\lambda_N^H$ for any $i$. Using $h_N \ge h_i$, we get $a_i\mu_i^R=\mu_i^S +\frac{\mu_i^S}{h_i} \ge \mu_i^S +\frac{\mu_i^S}{h_N} = a_N\{\mu_i^R\}^H > a_N\lambda_N^H.$ Therefore, $a_i\lambda_i^* = \min \left\{ \frac{f_i(x^*)}{\frac{d f_i}{d x}(x^*)}, a_i\mu_i^R\right\} > a_N\lambda_N^H$, for any $i$. Hence, $x^* < x^H$, which is a contradiction. Therefore, $x^H \ge x^*$ (equivalently, $\sum_{i=1}^{N}a_i \lambda_i^* \ge N a_N \lambda_N^H $) and $\frac{d f_i}{d x}(x^*) \ge \frac{d f_i}{d x}(x^H) >0$.

\noindent
\textbf{Step 2:} \textit{We show that $\sum_i f_i(x^*) \ge Nf_N(x^H)$ \& $\sum_i \lambda_i^* \ge N \lambda_N^H $}.  

Let $d \le N$ be the support for $\lambda^*$. Therefore, $\lambda_i^*= \min\left\{\frac{f_i(x^*)}{a_i\frac{d f_i}{d x}(x^*) }, \mu_i^R\right\}$ for  every $i \le d$,  , and $\lambda_i^*= 0$ for any $i>d$. Therefore, $\sum_{i=1}^{d}\frac{f_i(x^*)}{\frac{d f_i}{d x}(x^*)} \ge \sum_{i=1}^{d}a_i \lambda_i^* \ge N a_N \lambda_N^H = N \frac{f_N(x^H)}{\frac{d f_N}{d x}(x^H)} $. Using $\frac{d f_i}{d x}(x^*) \ge \frac{d f_N}{d x}(x^H) >0$ ( $\frac{d f_i}{d x}$ is independent of $i$), we get  $\sum_{i=1}^{d}f_i(x^*) \ge Nf_N(x^H)$.  Additionally, $\sum_{i=1}^{N}a_i \lambda_i^* \ge N a_N \lambda_N^H $ implies $\sum_{i=1}^{d}\lambda_i^* \ge \sum_{i=1}^{N}\frac{a_i}{a_N} \lambda_i^* \ge N \lambda_N^H $.

\noindent
\textbf{Step 3:} \textit{We show that $\Psi^{\Gamma}_{PNE} \ge \Psi^H_{PNE}$.}

Using $h_N \ge h_i$, we have $\mu_i^R=\frac{\mu_i^S}{h_i} \ge \frac{\mu_i^S}{h_N} = \{\mu_i^R\}^H > \lambda_N^H.$ Let $d_1\le d$ be the largest index of player satisfying $\frac{f_i(x^*)}{a_i\frac{d f_i}{d x}} > \lambda_N^H$. Since $\frac{f_N(x^*)}{a_N\frac{d f_N}{d x}(x^*)} < \lambda_N^H$, $d_1 < N$. Therefore, $\lambda_i^*$ satisfies,
\begin{equation}{\label{eq:step}}
   \lambda_i^* = \begin{cases}
\min\left\{\frac{f_i(x^*)}{a_i\frac{d f_i}{d x}(x^*)}, \mu_i^R\right\} > \lambda_N^H , \ \ \ \text{for } i \le d_1,  \\
    \frac{f_i(x^*)}{a_i\frac{d f_i}{d x}(x^*)} \le \lambda_N^H, \ \ \  \text{for } d_1 + 1 \le i \le d.
    \end{cases}
\end{equation}
 Hence,
\begin{align*}
    \Psi^{\Gamma}_{PNE}&= \mu_T^Sr^S + \sum_{i=1}^{d}\lambda_i^*f_i(x^*)\\
    &\overset{(1^*)}{=} \mu_T^Sr^S + \lambda_N^H\sum_{i=1}^{d}f_i(x^*) 
    + \sum_{i=1}^{d_1}(\lambda_i^*- \lambda_N^H)f_i(x^*) \\  &+ \sum_{i=d_1+1}^{d}(\lambda_i^*- \lambda_N^H)f_i(x^*) \\
    &\overset{(2^*)}{\ge} \mu_T^Sr^S  + \lambda_N^H N f_N(x^H)
    + \sum_{i=1}^{d_1}(\lambda_i^*- \lambda_N^H)f_{d_1+1}(x^*) \\  &+ \sum_{i=d_1+1}^{d}(\lambda_i^*- \lambda_N^H)f_{d_1+1}(x^*) 
\end{align*}     
\begin{align*}
        &= \mu_T^Sr^S  + \lambda_N^H N f_N(x^H)
    + f_{d_1+1}(x^*) \sum_{i=1}^{d}(\lambda_i^*- \lambda_N^H) \\
    &\overset{(3^*)}{\ge}  \mu_T^Sr^S  + \lambda_N^H N f_N(x^H)=\Psi_{PNE}^{H},
\end{align*}
where $(1^*)$ follows by adding and subtracting $\lambda_N^H\sum_{i=1}^{d}f_i(x^*)$. $(2^*)$ follows from $\sum_{i=1}^{d}f_i(x^*) \ge Nf_N(x^H)$ (Step 2),~\eqref{eq:step}, and the fact that $f_1(x^*) \ge \cdots \ge f_N(x^*)$. $(3^*)$ follows by recalling that $\sum_{i=1}^{d}\lambda_i^* \ge N \lambda_N^H $ (Step 2).

\textbf{Step 4:} \textit{We show that $\Psi^{\Gamma}_{PNE} \ge \Psi^H_{PNE}  \ge \mu_T^Sr^S +\frac{\mu_T^S- \overline{x}}{a_N}f_N(\overline{x})$.}

Let $\overline{x}$ be the unique maximizer of $f_i$.  From Lemma~\ref{lemma:homogeneous PoA}, $ \Psi^H_{PNE} = \Psi^H_{SW} = \max\{\Psi^H\} \ge \mu_T^Sr^S + \lambda_T^Rf_N(\mu_T^S - a_N\lambda_T^R)$ for any $\lambda_T^R$.  Choosing $\lambda^R_T = \frac{\mu_T^S-\overline{x}}{a_N}$,  we obtain the desired bounds.  
 \end{proof}

\medskip
\textit{Proof of Theorem~\ref{thm:thm3}:}
The global optimum of social welfare function is upper bounded by: 
\begin{align}{\label{eq:POA_num}}
    \Psi^{\Gamma}_{SW}&= \mu_T^Sr^S+ \max_{\lambda_i^R}\left\{ \sum_{i=1}^{N}\lambda_i^Rf_i(x)\right\} \nonumber\\ 
    &\le \mu_T^Sr^S + \max_{\lambda_i^R}\left\{\lambda_T^R\right\}\max_{x}\{f_i(x)\} \nonumber\\ 
     &\overset{(1^*)}{\le} \mu_T^Sr^S +  \mu_T^S{f_i(\overline{x})} \nonumber\\ 
      &\le \mu_T^S(r^S + r^R(\overline{x})(1-p(\overline{x}))),
\end{align}
where $(1^*)$ is obtained using the system constraint $x>0$ which implies $\mu_T^S \ge \sum_{i=1}^{N}a_i\lambda_i^R \ge \lambda_T^R$. 

From Lemma~\ref{lemma:homogeneous vs heterogeneous}, $\Psi^{\Gamma}_{PNE}$ is lower bounded by:
\begin{align}{\label{eq:POA_denom}}
    \Psi^{\Gamma}_{PNE} &\ge \mu_T^Sr^S +\frac{\mu_T^S- \overline{x}}{a_N}f_N(\overline{x}) \nonumber \\
    &= \frac{1}{a_N}\left(\overline{x}a_Nr^S   + (\mu_T^S- \overline{x})(r^S+r^R(\overline{x})(1-p(\overline{x})))\right) \nonumber \\
    & \ge \frac{1}{a_N}(\mu_T^S- \overline{x})(r^S+r^R(\overline{x})(1-p(\overline{x}))).
\end{align}
Using~\eqref{eq:POA_num} and~\eqref{eq:POA_denom}, we get, $PoA = \frac{\Psi_{SW}^{\Gamma}}{\Psi_{PNE}^{\Gamma}} \le \frac{\mu_T^S a_N}{\mu_T^S-\overline{x}}$.

Now we establish the bounds on $\eta_{TRI}$ and $\eta_{LI}$. Let $x_{PNE}$ and $x_{SW}$ be the slackness parameter corresponding to the PNE and social welfare, respectively. Recall that $\frac{d f_i}{d x} >0$ (Corollary~\ref{corollary2}), for $x \in \{x_{PNE}, x_{SW}\}$. 
Hence, using strict concavity of $f_i$, we have $ x_{PNE}, x_{SW} \in (0, \  \overline{x})$.
Therefore, $\mu_T^S-\overline{x} < \sum_{i=1}^{N}a_i\{\lambda_i^R\}_{PNE} < \mu_T^S$, and $\mu_T^S-\overline{x} < \sum_{i=1}^{N}a_i\{\lambda_i^R\}_{SW} < \mu_T^S$. Hence, $\eta_{TRI}$ and $\eta_{LI}$ are upper bounded by:

\begin{equation*}
 \eta_{TRI}=\frac{({\lambda_T^R})_{SW}}{{(\lambda_T^R})_{PNE}} < \frac{\mu_T^S a_N}{(\mu_T^S - \overline{x})a_1}, \text{and}
 \end{equation*}
 \begin{equation*}
    \eta_{LI}=\frac{(\sum_{i=1}^{N}a_i{\lambda_i^R})_{PNE}}{(\sum_{i=1}^{N}a_i{\lambda_i^R})_{SW}} < \frac{\mu_T^S}{\mu_T^S - \overline{x}}.
\end{equation*}
\oprocend

\footnotesize 
\bibliographystyle{IEEEtran}
\bibliography{mybib}

\begin{thebibliography}{10}
\providecommand{\url}[1]{#1}
\csname url@samestyle\endcsname
\providecommand{\newblock}{\relax}
\providecommand{\bibinfo}[2]{#2}
\providecommand{\BIBentrySTDinterwordspacing}{\spaceskip=0pt\relax}
\providecommand{\BIBentryALTinterwordstretchfactor}{4}
\providecommand{\BIBentryALTinterwordspacing}{\spaceskip=\fontdimen2\font plus
\BIBentryALTinterwordstretchfactor\fontdimen3\font minus
  \fontdimen4\font\relax}
\providecommand{\BIBforeignlanguage}[2]{{%
\expandafter\ifx\csname l@#1\endcsname\relax
\typeout{** WARNING: IEEEtran.bst: No hyphenation pattern has been}%
\typeout{** loaded for the language `#1'. Using the pattern for}%
\typeout{** the default language instead.}%
\else
\language=\csname l@#1\endcsname
\fi
#2}}
\providecommand{\BIBdecl}{\relax}
\BIBdecl

\bibitem{gupta2019achieving}
P.~Gupta, S.~D. Bopardikar, and V.~Srivastava, ``Achieving efficient
  collaboration in decentralized heterogeneous teams using common-pool resource
  games,'' in \emph{2019 IEEE 58th Conference on Decision and Control
  (CDC)}.\hskip 1em plus 0.5em minus 0.4em\relax IEEE, 2019, pp. 6924--6929.

\bibitem{haas2016secrets}
M.~Haas and M.~Mortensen, ``The secrets of great teamwork.'' \emph{Harvard
  Business Review}, vol.~94, no.~6, pp. 70--6, 2016.

\bibitem{burns2005mechanistic}
T.~Burns and G.~Stalker, ``Mechanistic and organic systems,'' \emph{Organ
  Behav}, vol.~2, pp. 214--225, 2005.

\bibitem{lunenburg2012mechanistic}
F.~C. Lunenburg, ``Mechanistic-organic organizations -- an axiomatic theory:
  Authority based on bureaucracy or professional norms,'' \emph{International
  Journal of Scholarly Academic Intellectual Diversity}, vol.~14, no.~1, pp.
  1--7, 2012.

\bibitem{keser1999strategic}
C.~Keser and R.~Gardner, ``Strategic behavior of experienced subjects in a
  common pool resource game,'' \emph{International Journal of Game Theory},
  vol.~28, no.~2, pp. 241--252, 1999.

\bibitem{hota2016fragility}
A.~R. Hota, S.~Garg, and S.~Sundaram, ``Fragility of the commons under
  prospect-theoretic risk attitudes,'' \emph{Games and Economic Behavior},
  vol.~98, pp. 135--164, 2016.

\bibitem{goodrich2007using}
M.~A. Goodrich, J.~L. Cooper, J.~A. Adams, C.~Humphrey, R.~Zeeman, and B.~G.
  Buss, ``Using a mini-{UAV} to support wilderness search and rescue: Practices
  for human-robot teaming,'' in \emph{Safety, Security and Rescue Robotics,
  2007. SSRR 2007. IEEE International Workshop on}.\hskip 1em plus 0.5em minus
  0.4em\relax IEEE, 2007, pp. 1--6.

\bibitem{JP-VS-etal:12t}
J.~Peters, V.~Srivastava, G.~Taylor, A.~Surana, M.~P. Eckstein, and F.~Bullo,
  ``Human supervisory control of robotic teams: Integrating cognitive modeling
  with engineering design,'' \emph{IEEE Control System Magazine}, vol.~35,
  no.~6, pp. 57--80, 2015.

\bibitem{VS-RC-CL-FB:11zc}
V.~Srivastava, R.~Carli, C.~Langbort, and F.~Bullo, ``Attention allocation for
  decision making queues,'' \emph{Automatica}, vol.~50, no.~2, pp. 378--388,
  2014.

\bibitem{PG-VS:18d}
P.~Gupta and V.~Srivastava, ``Optimal fidelity selection for human-in-the-loop
  queues using semi-{M}arkov decision processes,'' in \emph{American Control
  Conference}, Philadelphia, PA, Jul. 2019, pp. 5266--5271.

\bibitem{marden2018game}
J.~R. Marden and J.~S. Shamma, ``Game theory and control,'' \emph{Annual Review
  of Control, Robotics, and Autonomous Systems}, vol.~1, pp. 105--134, 2018.

\bibitem{arslan2007autonomous}
G.~Arslan, J.~Marden, and J.~Shamma, ``Autonomous vehicle-target assignment: A
  game-theoretical formulation,'' \emph{Journal of Dynamic Systems Measurement
  and Control-Transactions of the Asme}, vol. 129, 09 2007.

\bibitem{basar1999dynamic}
T.~Ba{\c{s}}ar and G.~J. Olsder, \emph{Dynamic Noncooperative Game
  Theory}.\hskip 1em plus 0.5em minus 0.4em\relax SIAM, 1999, vol.~23.

\bibitem{roughgarden2009intrinsic}
T.~Roughgarden, ``Intrinsic robustness of the price of anarchy,'' in
  \emph{Proceedings of the forty-first annual ACM symposium on Theory of
  computing}, 2009, pp. 513--522.

\bibitem{marden2014generalized}
J.~R. Marden and T.~Roughgarden, ``Generalized efficiency bounds in distributed
  resource allocation,'' \emph{IEEE Transactions on Automatic Control},
  vol.~59, no.~3, pp. 571--584, 2014.

\bibitem{deori2018price}
L.~Deori, K.~Margellos, and M.~Prandini, ``Price of anarchy in electric vehicle
  charging control games: When nash equilibria achieve social welfare,''
  \emph{Automatica}, vol.~96, pp. 150--158, 2018.

\bibitem{paccagnan2019utility}
D.~{Paccagnan}, R.~{Chandan}, and J.~R. {Marden}, ``Utility design for
  distributed resource allocation - part {I}: Characterizing and optimizing the
  exact price of anarchy,'' \emph{IEEE Transactions on Automatic Control}, pp.
  1--1, 2019.

\bibitem{gao2014modeling}
F.~Gao, M.~L. Cummings, and E.~T. Solovey, ``Modeling teamwork in supervisory
  control of multiple robots,'' \emph{IEEE Transactions on Human-Machine
  Systems}, vol.~44, no.~4, pp. 441--453, 2014.

\bibitem{hong2016human}
A.~Hong, ``Human-robot interactions for single robots and multi-robot teams,''
  Ph.D. dissertation, University of Toronto, 2016.

\bibitem{srivastava2014knapsack}
V.~Srivastava and F.~Bullo, ``Knapsack problems with sigmoid utilities:
  Approximation algorithms via hybrid optimization,'' \emph{European Journal of
  Operational Research}, vol. 236, no.~2, pp. 488--498, 2014.

\bibitem{mekdeci2009modeling}
B.~Mekdeci and M.~Cummings, ``Modeling multiple human operators in the
  supervisory control of heterogeneous unmanned vehicles,'' in
  \emph{Proceedings of the 9th Workshop on Performance Metrics for Intelligent
  Systems}.\hskip 1em plus 0.5em minus 0.4em\relax ACM, 2009, pp. 1--8.

\bibitem{le1997theory}
P.~Le~Gall, ``The theory of networks of single server queues and the tandem
  queue model,'' \emph{International Journal of Stochastic Analysis}, vol.~10,
  no.~4, pp. 363--381, 1997.

\bibitem{thomopoulos2012fundamentals}
N.~T. Thomopoulos, \emph{Fundamentals of queuing systems: statistical methods
  for analyzing queuing models}.\hskip 1em plus 0.5em minus 0.4em\relax
  Springer Science \& Business Media, 2012.

\bibitem{altman2005applications}
E.~Altman, ``Applications of dynamic games in queues,'' in \emph{Advances in
  dynamic games}.\hskip 1em plus 0.5em minus 0.4em\relax Springer, 2005, pp.
  309--342.

\bibitem{xia2014service}
L.~Xia, ``Service rate control of closed jackson networks from game theoretic
  perspective,'' \emph{European Journal of Operational Research}, vol. 237,
  no.~2, pp. 546--554, 2014.

\bibitem{hota2018controlling}
A.~R. Hota and S.~Sundaram, ``Controlling human utilization of failure-prone
  systems via taxes,'' \emph{arXiv preprint arXiv:1802.09490}, 2018.

\bibitem{ostrom1994rules}
E.~Ostrom, R.~Gardner, J.~Walker, and J.~Walker, \emph{Rules, Games, and
  Common-Pool Resources}.\hskip 1em plus 0.5em minus 0.4em\relax University of
  Michigan Press, 1994.

\bibitem{voorneveld2000best}
M.~Voorneveld, ``Best-response potential games,'' \emph{Economics letters},
  vol.~66, no.~3, pp. 289--295, 2000.

\bibitem{dubey2006strategic}
P.~Dubey, O.~Haimanko, and A.~Zapechelnyuk, ``Strategic complements and
  substitutes, and potential games,'' \emph{Games and Economic Behavior},
  vol.~54, no.~1, pp. 77--94, 2006.

\bibitem{jensen2009stability}
M.~K. Jensen, ``Stability of pure strategy nash equilibrium in best-reply
  potential games,'' \emph{University of Birmingham, Tech. Rep}, 2009.

\bibitem{cassandras2009introduction}
C.~G. Cassandras and S.~Lafortune, \emph{Introduction to Discrete Event
  Systems}.\hskip 1em plus 0.5em minus 0.4em\relax Springer Science \& Business
  Media, 2009.

\bibitem{jensen2010aggregative}
M.~K. Jensen, ``Aggregative games and best-reply potentials,'' \emph{Economic
  Theory}, vol.~43, no.~1, pp. 45--66, 2010.

\bibitem{schipper2004pseudo}
\BIBentryALTinterwordspacing
B.~Schipper, ``Pseudo-potential games,'' University of Bonn, Germany, Tech.
  Rep., 2004, working paper. [Online]. Available: \url{available at
  citeseer.ist.psu.edu/schipper04pseudopotential.html}
\BIBentrySTDinterwordspacing

\bibitem{burden19852}
R.~Burden and J.~Faires, ``The bisection method,'' \emph{Numerical Analysis},
  pp. 48--56, 2011.

\bibitem{luenberger1984linear}
D.~G. Luenberger and Y.~Ye, \emph{Linear and Nonlinear Programming}.\hskip 1em
  plus 0.5em minus 0.4em\relax Springer, 1984, vol.~2.

\bibitem{berge1997topological}
C.~Berge, \emph{Topological Spaces: Including a Treatment of Multi-valued
  Functions, Vector Spaces, and Convexity}.\hskip 1em plus 0.5em minus
  0.4em\relax Courier Corporation, 1997.

\end{thebibliography}

\end{document}